\documentclass{amsart}

\usepackage{amsmath,amsfonts,hyperref,mdwlist,graphicx,enumitem}
\usepackage{amssymb}
\usepackage[cmtip,all]{xy}
\usepackage[latin1]{inputenc}
\newtheorem{theorem}{Theorem}[section]
\newtheorem{lemma}[theorem]{Lemma}
\newtheorem{proposition}[theorem]{Proposition}

\theoremstyle{definition}

\theoremstyle{remark}
\newtheorem{remark}[theorem]{Remark}

\numberwithin{equation}{section}

\newcommand{\alg}{\textnormal{\bf{alg}}\,}
\newcommand{\vect}{\textnormal{\bf{vect}}\,}
\newcommand{\am}{^{(+)}}\newcommand{\wt}{\widetilde}
\newcommand{\mro}{\textnormal{R}} %operator to right
\newcommand{\reg}{\mathcal{R}\textnormal{eg}\,}\newcommand{\sld}{sl_2}
\newcommand{\corf}{\mathbb{F}}\newcommand{\corfd}{\corf\cdot}% field
\newcommand{\jor}{\mathfrak{J}}\newcommand{\rad}{\mathcal{N}}\newcommand{\radd}{\mathcal{N}\,^2}
\newcommand{\bim}{\mathcal{M}}\newcommand{\obim}{\mathcal{M}^{\textnormal{op}}}\newcommand{\biml}{\mathcal{L}}
\newcommand{\clasekmj}{\mathcal{K}(\mathfrak{M},\jor)} \newcommand{\mj}{\mathfrak{M}(\jor)}

\newcommand{\heraa}{\mathcal{H}(\mathcal{A},\ast)}%notacion para algebras hermitianas
\newcommand{\gra}{\Gamma}%notacion algebra de Grassman
\newcommand{\alga}{\mathcal{A}}\newcommand{\algc}{\mathcal{C}}\newcommand{\algs}{\mathcal{S}}
\newcommand{\algb}{\mathcal{B}}

\newcommand{\alMnm}{\mathcal{M}_{n|m}(\corf)}
\newcommand{\alMnn}{\mathcal{Q}_{n}(\corf)}
\newcommand{\osp}{\textnormal{Osp}}%% superinvolutions

\newcommand{\kap}{\mathcal{K}_{3}} \newcommand{\kapun}{\mathcal{K}_{3}\oplus\corfd 1}
\newcommand{\kac}{\mathcal{K}_{10}}
\newcommand{\algdt}{\mathcal{D}_{t}}\newcommand{\algdc}{\mathcal{D}_{0}}

\newcommand{\alQn}{\mathcal{Q}_{n}(\corf)}
\newcommand{\suMnm}{\mathcal{M}_{n|m}(\corf)\am}\newcommand{\suMuu}{\mathcal{M}_{1|1}(\corf)\am}
\newcommand{\suMnn}{\mathcal{Q}_{n}(\corf)\am}
\newcommand{\suMn}{\alQn\am}
\newcommand{\josp}{\jor\textnormal{osp}}

\newcommand{\bav}{\bar{v}}
\newcommand{\wv}{\widetilde{v}}\newcommand{\wu}{\widetilde{u}}
\newcommand{\ww}{\widetilde{w}}\newcommand{\we}{\widetilde{e}_1}\newcommand{\wf}{\widetilde{e}_2}
\newcommand{\wx}{\widetilde{x}}\newcommand{\wy}{\widetilde{y}}

\newcommand{\Jpn}{\textnormal{JP}_n(\corf)}

\begin{document}
\title{Wedderburn principal theorem for Jordan superalgebras I.}
\author{G\'omez-Gonz\'alez F.A.}
\address{Instituto de Matematicas, Universidad de Antioquia, Colombia,} %\newline\indent{Instituto de Matematicas e Estatistica, Universidade de Sao Paulo, Brazil}}
%\dedicatory{{\bf{Preprint}}}
%\email{fagogo@ciencias.udea.edu.co, fagogo@ime.usp.br}
\thanks{The author is thankful to Prof. Ivan Shestakov for his suggestion to solve the problem considered in this paper as a part of the author's Doctoral Thesis and for other valuable advises.}
\thanks{The author was partially supported by CAPES/CNPq IEL Nacional, Brazil, and Universidad de Antioquia,  Colombia}
\begin{abstract}
%RC

We consider finite dimensional Jordan superalgebras $\jor$ over an algebraically closed field of characteristic 0, with solvable radical
$\rad$
such that
$\radd=0$
and
$\jor/\rad$
is a simple Jordan superalgebra of one of the following types:
Kac $\kac$,
Kaplansky
$\mathcal{K}_3$
superform or 
$\algdt$.

We prove that an analogue of the
\emph{Wedderburn Principal Theorem (WPT)}
holds if certain restrictions on the types of irreducible subsuperbimodules of
$\rad$
are imposed,
where
$\rad$
is considered as a
$\jor/\rad$-superbimodule.
Using counterexamples, it is shown that the imposed restrictions are essential.
\end{abstract}

\maketitle


%\tableofcontents
\section{Introduction}
%RC
In 1892, T.~Mollien~\cite{Mol} proved that for any finite-dimensional associative algebra
$\alga$
with nilpotent radical
$\rad$
over the complex field there exists a subalgebra
$\algs\subseteq \alga$
such that
$\algs\cong\alga/\rad$
and
$\alga=\algs\oplus\rad$.
This result was generalized in 1905 by J.~H.~Maclagan-Wedderburn~\cite{Wed1}
for all finite dimensional associative algebras over an arbitrary field.
This result is known as the Wedderburn's Principal Theorem (WPT).
Analogues of the WPT were proved for finite-dimensional alternative algebras by R.~D.~Schafer
\cite{Sch1},
and for finite-dimensional Jordan algebras by A. Albert, Penico, Askinuze, and Taft~\cite{Alb,Ask,Taf,Pen}. Thus it is natural to try to extend this result to superalgebras.

In the case of finite dimensional alternative superalgebras
$\alga$
over a  field of characteristic zero, Pisarenko~\cite{Pis} proved an analogue to the WPT.
He proved that the theorem holds if some restrictions are imposed over summands in the semisimple superalgebra
$\alga/\rad$.
It was also shown with counter-examples that the restrictions are essential.

In the current paper,
we consider finite dimensional Jordan superalgebras
$\alga$
over a field of characteristic zero with radical
$\rad$
such that
$\radd=0$
and
$\alga/\rad$
is a simple Jordan superalgebra of one of the following types:
Kac $\kac$,
Kaplansky $\kap$, 
superform or 
$\algdt$.

The cases of simple quotients of the types
$\mathcal K_{10}$, superform, $D_t$, $\mathcal K_{3}$
are considered.
It's proved that a Wedderburn decomposition is possible with certain essential restrictions

\medskip

This paper is organized as follows. 
In Section 2, the basic examples of Jordan superalgebras are given.
Sections 3-6 contain the proof of the Main Theorem. 
In Section 3, the necessary reductions are done.
Sections 4-6 are devoted to prove of theorem.
Section 7, the main theorem is enounced.

\medspace
Note that the cases
$\suMnm$ and
$\josp_{n|2m}(\corf)$, 
are considered in~\cite{yo1} and \cite{yo2} respectively.
%RC
The other cases, when
$\jor/\rad$
is isomorphic to

$\Jpn, \suMnn$,
$\kap\oplus\kap\oplus\cdots\oplus\kapun$,
and Kantor superalgebra, are to be considered in the next paper.

\medskip
%RC
We also stress that the Main Theorem implies that the second cohomology group
$H^2(\jor,\rad)$
is not trivial for some simple Jordan superalgebra
$\jor$
and some irreducible
$\jor$-superbimodule
$\rad$.
This gives one more subject of interest to be considered in future papers.

\section{Jordan superalgebras, definition and some examples}

%RC
Throughout the paper, all algebras are considered over an algebraically closed field of characteristic zero $\corf$.

\medskip
%RC
Recall that an algebra
$\alga$
is said to be a \textit{superalgebra} if it is a direct sum
$\alga=\alga_0\dotplus\alga_1$
of vector spaces satisfying the relation
$\alga_{i}\alga_{j}\subseteq \alga_{i+j(\text{\rm mod }2)}$,
i.e.
$\alga$
is a
$\mathbb{Z}_2$-graded algebra.
For an element
$a\in\alga_i,\,i=0,1$,
the number
$|a|=i$
denotes a \textit{parity} of $a$.

\medskip

%RC
Let
$\Gamma=\alg\langle\,1,e_i,\,\, i\in\mathbb{Z}^{+} | e_ie_j+e_je_i=0\,\rangle$
be the Grassmann algebra.
Then
$\gra=\gra_0\dotplus\gra_1$,
where
$\gra_0$ and $\gra_1$
are the spans of all monomials of even and odd lengths, respectively.
It is not difficult to see that
$\gra$
has a superalgebra structure.

\medskip

%RC
For a superalgebra
$\alga=\alga_0\dotplus\alga_1$,
we define the \textit{Grassmann envelope of}
$\alga$
as follows:
$\gra(\alga)=\gra_0\otimes\alga_0\dotplus\gra_1\otimes\alga_1$.
Assuming that 
$\mathfrak{M}$
is a homogeneous variety of algebras.
The superalgebra
$\alga$
is said to be an
$\mathfrak{M}$-\textit{superalgebra} if the Grassmann envelope
$\gra(\alga)$
lies in
$\mathfrak{M}$.
Following this definition, one can consider associative, alternative, Lie, Jordan, etc. superalgebras.

\medskip
We recall that an algebra 
$\jor$ 
is said a Jordan algebra if its multiplication satisfies  the identity
$ab=ba$ of commutativity and the
Jordan identity 
$(a^2b)a=a^2(ba)$.
In this paper, we consider algebras over a field characteristic zero. 
Thus,  the Jordan identity is equivalent to its complete linearization 
$$((ac)b)d+((ad)b)c+((cd)b)a=(ac)(bd)+(ad)(bc)+(cd)(ba).$$

\medskip%RC 
An associative superalgebra is just a
$\mathbb{Z}_2$-graded associative algebra,
but it is not the case in general terms. It is easy to see that a Jordan superalgebra it is not a Jordan algebra.
%RC:
One can verify that a superalgebra
$\jor=\jor_0\dotplus\jor_1$
is a Jordan superalgebra iff it satisfies the superidentities
\begin{eqnarray}
& a_ia_j=(-1)^{ij}a_ja_i, \label{idsupercom} \\
&\begin{gathered}
((a_ia_j)a_k)a_l+(-1)^{l(k+j)+kj}((a_ia_l)a_k)a_j+(-1)^{i(j+k+l)+kl}((a_ja_l)a_k)a_i=\\
=(a_ia_j)(a_ka_l)+(-1)^{l(k+j)}(a_ia_l)(a_ja_k)+(-1)^{jk}(a_ia_k)(a_ja_l) \label{idsupjordan}
\end{gathered}
\end{eqnarray}
%RC
for homogeneous elements
$a_t\in\jor_t,\,t\in\{i,j,k,l\}$.

%RC:
We stress that, in view of the restriction on the characteristic of ground field,
superidentity~\eqref{idsupercom} yields that the Jordan superalgebra
$\jor=\jor_0\dotplus\jor_1$
is a
($\mathbb{Z}_2$-graded)
Jordan algebra iff
$(\jor_1)^2=0$.

Throughout the paper, we denote by $\dotplus$ a direct sum of vector space, by $+$ 
denote a sum of vector space and by $\oplus$ we denote a direct sum of superalgebras.

\medskip
\textbf{Some examples of Jordan Superalgebras.}

\medskip
%RC
Let
$\alga$
be an associative superalgebra with multiplication
$ab$.
We define on the vector space
$\alga$
a new multiplication
$a\circ b=\frac{1}{2}(ab+(-1)^{|a||b|}ba)$
for
$a,\,b\in\alga_0\cup\alga_1$.
It is not hard to verify that
$\alga$
gains a structure of Jordan superalgebra with respect to the defined multiplication.
We denote this superalgebra by
$\alga\am$.

\medskip
C.T.C~ Wall~\cite{Wal} proved that every associative simple finite-dimensional superalgebra over an algebraically closed field
$\corf$
is isomorphic to one of the following associative superalgebras:
\begin{itemize}
\item[{\rm (i)}] $\alga=\alMnm$,\quad $\alga_0=\Big\{\Big(\begin{array}{ll} a& 0\\ 0 & d \end{array}\Big)\Big\},\quad \alga_1=\Big\{\Big(\begin{array}{ll} 0& b\\ c & 0\end{array}\Big)\Big\},$
\item[{\rm (ii)}] $\alga=\alMnn=\mathcal{Q}(n)$,\quad $\alga_0=\Big\{\Big(\begin{array}{ll} a& 0\\ 0 & a \end{array}\Big)\Big\},\quad \alga_1=\Big\{\Big(\begin{array}{ll} 0& h\\ h & 0\end{array}\Big)\Big\}.$
\end{itemize}
where
$a, \, h\in\mathcal{M}_{n}(\corf)$, $d\in\mathcal{M}_{m}(\corf)$, $b\in\mathcal{M}_{n\times m}(\corf)$, $c\in\mathcal{M}_{m\times n}(\corf)$.

\medskip\noindent{\bf{(I)}}
%RC
Applying the multiplication
``$\circ$'' to the associative superalgebras
$\alMnn$ and $\alMnm$,
we get the Jordan superalgebras
$\suMnn$ and  $\suMnm$
respectively.

\medskip\noindent{\bf{(II)}}
%RC
Let
$\alga$
be an associative superalgebra.
A graded linear mapping
$\ast:\alga\longrightarrow \alga$
is called \textit{superinvolution} if
$(a^{\ast})^{\ast}=a$
and
$(ab)^{\ast}=(-1)^{|a||b|}b^{\ast}a^{\ast}$.
\noindent
By
$\heraa$
denote the set of symmetric elements of
$\alga$
relative to
$\ast$.
Then
$\heraa$
is a Jordan superalgebra such that
$\heraa\subseteq \alga\am$.

\noindent
%RC
Let
$I_n,\,I_m$
be the identity matrices of order $n$ and $m$ respectively,
$t$
be the transposition and
$$
U=-U^t=-U^{-1}=\Big(\begin{array}{ll}0& -I_m\\ I_m& 0\end{array}\Big).
$$

\noindent
%RC
Consider linear mappings
$$
\osp:\bim_{n|2m}(\corf)\longrightarrow\bim_{n|2m}(\corf) \textnormal{ and } \sigma:\alMnn\longrightarrow \alMnn
$$
given by
\begin{align}
\Big(\begin{array}{ll}a&b\\c&d\end{array}\Big)^{\osp}&=\Big(\begin{array}{ll}I_n&0\\0&U\end{array}\Big)\Big(\begin{array}{ll}a^t&-c^t\\b^t&d^t\end{array}\Big)\Big(\begin{array}{ll}I_n&0\\ 0&U^{-1}\end{array}\Big),\nonumber\\
\Big(\begin{array}{ll}a&b\\ c& d\end{array}\Big)^{\sigma}&=\Big(\begin{array}{ll}d^t&-b^t\\ c^t& a^t\end{array}\Big).
\end{align}
\noindent
It is easy to check that
$\osp$
and
$\sigma$
are superinvolutions and its  Jordan superalgebras are
$\mathcal{H}(\bim_{n|2m}(\corf),\osp)$
and
$\mathcal{H}(\alMnn,\sigma)$.
We denote these superalgebras by
$\josp_{n|2m}{(\corf)}$
and
$\Jpn$
respectively.

\medskip
%RC
One also may consider the following Jordan superalgebras.

\medskip\noindent{\bf{(III)}}
%RC
The  4-dimensional 1-parametric family
$\algdt=(\corfd e_1+\corfd e_2)\dotplus(\corfd x+\corfd y)$,
with nonzero products given by
$e_i^2=e_i$, $e_ix=xe_i=\frac{1}{2}x$,
$e_iy=ye_i=\frac{1}{2}y$,
$xy=-yx=e_1+te_2$.
The superalgebra
$\algdt$
is simple for
$t\neq 0$.

\medskip\noindent{\bf{(IV)}}
%RC
The non unital 3-dimensional Kaplansky superalgebra
$\mathcal{K}_3=\corfd e\dotplus(\corfd x+\corfd y)$,
with nonzero products
$ex=xe=\frac{1}{2}x$, $ey=ye=\frac{1}{2}y$, $xy=-yx=e$.
The superalgebra
$\kap$
is simple.

\medskip\noindent{\bf{(V)}}
%RC
Let $V=V_0\oplus V_1$ be a vector superspace.
We say that a bilinear mapping
$f:V\times V\longrightarrow \corf$
is a superform if
$f$
is symmetric over
$V_0$,
skew-symmetric over
$V_1$,
and satisfies
$f(V_0,V_1)=0$.
Consider a superalgebra
$\jor=(\corfd 1\oplus V_0)\dotplus V_1$
with the unit $1$
and the multiplication
$v\cdot w=f(v,w)\cdot 1$, ($v,w\in V$).
If
$f$
is a non-degenerate superform and
$\textnormal{dim }V_0> 1$,
then
$\jor$
is a simple Jordan superalgebra.

\medskip\noindent{\bf{(VI)}}
%RC
The introduced by Kac 10-dimensional superalgebra
$\mathcal{K}_{10}$
is a simple Jordan superalgebra.
A detailed description of
$\mathcal{K}_{10}$
is given in Section \ref{seckac}.

\medskip\noindent{\bf{(VII)}}
I. Kantor \cite{Kan} defined a simple Jordan superalgebra structure in the finite-dimensional Grassmann algebra
generated by
$e_1,\ldots,e_n$. 

\medskip
V. Kac \cite{Kac},
proved that every simple finite-dimensional Jordan superalgebra over
$\corf$
is isomorphic to one of the superalgebras
$\suMnm,$ $\suMnn,$ $\josp_{n|2m}(\corf)$, $ \Jpn,\,\algdt,\,\mathcal{K}_3,\,\mathcal{K}_{10}$,
a superalgebra of superform or a Kantor superalgebra.

\medskip
A
$\jor$-superbimodule
$\bim=\bim_0\dotplus\bim_1$
is called a Jordan superbimodule if the corresponding split null extension
$\mathcal{E}=\jor\oplus\bim$
is a Jordan superalgebra \cite{Jac1}.
Recalling that the split null extension is a direct sum
$\jor\oplus\bim$
of vector spaces with a multiplication that extends the multiplication in
$\jor$
through the action of $\jor$ on $\bim$,
while the product of two arbitrary elements in
$\bim$
is zero.

%RC
Let
$\bim$
be a
$\jor$-superbimodule.
The opposite superbimodule
$\obim=\obim_0\dotplus\obim_1$
is defined by the conditions
$\obim_0=\bim_1,\,\obim_1=\bim_0$,
and by the following action of
$\jor$ over $\obim$:
$a\cdot m^{\textnormal{op}}=(-1)^{|a|}(am)^{\textnormal{op}}, \, m^{\textnormal{op}}\cdot a=(ma)^{\textnormal{op}}$
for all
$a\in\jor_0\cup\jor_1,\,m\in\obim_0\cup\obim_1$.
Whenever
$\bim$
is a Jordan
$\jor$-superbimule,
$\obim$ is a Jordan one as well.

\medskip Let $\alga=\jor$ as a vector superspace and let $am$, $ma$ with $m\in\jor$, $a\in\alga$ be the products as defined in the superalgebra $\jor$. It is easy to see that $\alga$ has a natural structure of $\jor$-superbimodule. We call $\alga$ a {\it{regular superbimodule}}.

\medskip
The irreducible superbimodules over the Jordan superalgebras of superform,
$\josp_{n|2m}(\corf)$, $\Jpn$, $\suMnm,$ were classified by C. Martinez and E. Zelmanov \cite{zelmar2}.
E. Zelmanov, C. Martinez and  I. Shestakov \cite{zelmarshe}, classified the irreducible superbimodules for Jordan superalgebras
$\suMn$. Irreducible superbimodules for Jordan superalgebra
$\algdt$ and $\mathcal{K}_3$
were  classified by C. Martinez and E. Zelmanov \cite{zelmar3} and independently by M. Trushina in \cite{Tru}.
C. Martinez and I. Shestakov \cite{marshe}, classified the irreducible superbimodules over the Jordan superalgebra 
$\suMuu$ and 
Shtern classified the irreducible superbimodules over Jordan superalgebras of type $\mathcal{K}_{10}$,
and Kantor superalgebra $\gamma (e_1,\ldots,e_n), \, n\geq 4 $ \cite{Sht1}.

\medskip
\noindent \textbf{The Peirce decomposition} 
Recall, that if $\jor$ is a Jordan (super)algebra with unity $1$, and  $\{e_1,\ldots,e_n\}$ is a 
set of pairwise orthogonal idempotents such that $1=\sum_{i=1}^n e_i$, then $\jor$ 
admits Peirce decomposition \cite{zelmar2},  
it is
$$\jor=\displaystyle{\biggl( \bigoplus_{i=1}^n\jor_{ii} \biggr) \bigoplus \biggl( \bigoplus_{i<j}\jor_{ij}\biggr) }, $$
where 
$\jor_{ii}=\displaystyle{\{\,x\in\jor: \quad e_ix=x,\}}$ and 
$\displaystyle{\jor_{ij}=\{\,x\in\jor: \quad e_ix=\frac{1}{2}x,\quad e_jx=\frac{1}{2}x\,\}}$, if $i\neq j$
are the Peirce components of $\jor$ relative to the idempotents $e_i$, and $ e_j$, 
moreover the following relations hold when $i\neq k,l; j\neq k,l$
\begin{align*}
&\jor_{ij}^2\subseteq \jor_{ii}+\jor_{jj},\quad \jor_{ij}\cdot\jor_{jk}\subseteq \jor_{ik}, \quad  
\jor_{ij}\cdot\jor_{kl}=0.
\end{align*}

%%%%%%%%%%%%%%%%%%%%%%%%%%%%%%%%%%%%%%%%%%%%%%%%%%%%%%%%%%%%%%%%%%%%%%%%%%%%%%%%%%%%%%%%%%%%%%%%%%%%%%%%%%%%%%%%%%%%%%%%%%%%%%%%%%%%%%%%%%%%%%%%%%%%%%%%%%%%%%%%%%%%%%%%%%%%%%%%%%%%%%%%%%%%%%%%%%%%%%%%%%%%%%%%%%%%%%%%%%%%%%%%%%%%%%%%%%%%%%%%%%%%%%%%%%%%%%%%%%%%%%%%%%%%%%%%%%%%%%%%%%%%%%%%%%%%%%%%%%%%%%%%%%%%%%%%%%%%%%%%%%%%%%%%%%%%%%%555

\section{Preliminary Reductions for WPT}
%RC
As in the case of Jordan algebras, we can make some restrictions before the main proof.
To start we prove the following proposition.

\begin{proposition}\label{pro1}
%RC
 Let $\jor$ be a Jordan superalgebra without 1 and with radical $\rad$.

If the WPT is valid for $\jor^{\#}$, then it is also valid for $\jor$.
\end{proposition}

\begin{proof}
%RC
Let
$\jor$
be a Jordan superalgebra without~$1$ and radical $\rad$.
Consider
$\jor^{\#}=\jor\oplus\corfd 1$.
It is clear that
$\rad(\jor)=\rad(\jor^{\#})=\rad$
and
$\jor^{\#}/\rad=(\jor/\rad)^{\#}=\jor/\rad\oplus\corfd \bar{1}$.
By the condition, there exists
$\algs_1\subseteq\jor^{\#}$, $\algs_1\cong\jor^{\#}/\rad\cong(\jor/\rad)^{\#}$, $\algs_1\cap\rad=(0)$,
$\jor^{\#}=\algs_1\oplus \rad$.
Denote
$\algs=\algs_1\cap\jor$,
then
$\algs\cap\rad=(0)$.
Let us show that
$\algs\oplus\rad=\jor$.
Take
$a\in\jor$, then
$a=s_1+n$, $s_1\in\algs_1, \,n\in \rad$.
But
$s_1=a-n\in\jor$.
Hence,
$s_1\in\jor\cap\algs_1=\algs$
and
$a\in \algs\oplus \rad$.
Finally,
$\algs\cong\algs/(\algs\cap\rad)\cong(\algs\oplus\rad)/\rad\cong\jor/\rad$.
\end{proof}

\medskip
Let
$\jor$
be a unital Jordan superalgebra of
$\textnormal{dim }\jor=n$.
Assume that for any unital Jordan superalgebra of dimension less that
$n$
the WPT is true.
A base for induction is
$\dim_{\corf}\jor=1$, $\jor=\corfd 1$.

\begin{proposition}\label{pro2}
%RC
Let
$\jor/\rad=\jor_1\oplus\cdots\oplus\jor_k$,
where
$\jor_i$
are unital simple Jordan superalgebras with
$\rad(\jor_i)=0$.
If
$k>1$,
then the WPT is true for
$\jor$.
\end{proposition}

\begin{proof}
%RC
Denote by
$e_i$
the identity elements in
$\jor_i$.
Then (by Jordan algebras results) there are orthogonal idempotents
$f_i\in\jor$
such that
$e_i=f_i+\rad$, $i=1,2,\ldots,k$.
Consider
$\jor_1(f_i)=\{f_i,\jor,f_i\} $,
then
$\jor_1(f_i)/(\jor_1(f_i)\cap\rad)\cong\jor_i $.
By virtue of
$\rad(\jor_i)=0$,
we have the inclusion
$\rad(\jor_1(f_i))\subseteq \jor_1(f_i)\cap\rad$.
Since the inverse inclusion is obvious, we have the equality
$\rad(\jor_1(f_i))=\jor_1(f_i)\cap\rad$.
If
$k>1$,
then
$\dim\jor_1(f_i) < \dim\jor$
and by the inductive hypothesis, there exists
$\algs_i\subseteq\jor_1(f_i)$, $\algs_i\cong\jor_i/(\rad\cap\jor_i)$.
Note that
$\algs_i\cdot\algs_j=0$.
Further,
$\algs=\algs_1\oplus\cdots\oplus\algs_k$
is a direct sum and
$\algs\cong\jor_1\oplus\cdots\oplus\jor_k$.
\end{proof}

\medskip
%RC
Now by the Zelmanov Theorem \cite{Zelss},
in the case of characteristic zero,
it is sufficient to prove the WPT for unital finite dimensional Jordan superalgebras $\jor$
satisfying one of the following conditions:

\begin{enumerate}
\item
$\jor/\rad$
is simple unital;
\item
$\jor/\rad=(\mathcal{K}_3\oplus\mathcal{K}_3\oplus\cdots\oplus\kap)\oplus\corfd 1$,
where
$\mathcal{K}_3$
is the Kaplansky superalgebra.
\end{enumerate}

\begin{theorem}\label{teoredu}
Let
$\jor$
be a finite dimensional semisimple Jordan superalgebra, i.e 
$\rad(\jor)=0$,
where
$\rad$
is the solvable radical.
Let 
$\mathfrak{M}(\jor)$ be a class 
of finite dimensional Jordan
$\jor$-superbimodules $\rad$ such that $\mathfrak{M}(\jor)$ 
is closed with respect to subsuperbimodules and homomorphic images.
Denote by
$\mathcal{K}(\mathfrak{M},\jor)$
the class of finite dimensional Jordan superalgebras
$\alga$
that satisfy the following conditions:
\begin{enumerate}
\item
$\alga/\rad(\alga)\cong \jor$,
\item
$\rad(\alga)^2=0$,
\item
$\rad(\alga)$
considered as a
$\jor$-superbimodule,
in 
$\mj$.
\end{enumerate}
Then if WPT is true for all superalgebras
$\algb\in\clasekmj$
with the restriction that the radical 
$\rad(\algb)$
 is an irreducible
$\jor$-superbimodule,
then it is true for all superalgebras
$\alga$
from
$\clasekmj$.
\end{theorem}

\begin{proof}
We use the induction on
$\dim\alga$.
The  base of induction is provide by the case 
$\dim\alga=\dim\jor$, so $\alga=\jor$, $\rad(\alga)=0$.
Assume that the theorem is true for all Jordan superalgebras
$\algb\in\clasekmj$
with
$\dim\algb\le \dim\alga$.
Let us set  by 
$\rad=\rad(\alga)$.
If
$\rad$
is an irreducible  $\jor$-superbimodule, then the theorem is true by the conjecture.
Suppose that
$\rad$
is not irreducible, then let us take a minimal
$\jor$-superbimodule
$\bim$ contained in $\rad$.
Since that
$\alga$
is unital
$\jor\bim=\alga\bim=\bim$,
therefore
$\rad$
is irreducible.
Observe that
$\rad/\bim\neq 0$,
otherwise
$\rad=\bim$
would be irreducible.
We see that
$\frac{\alga/\bim}{\rad/\bim}\cong\alga/\rad\cong\jor$.

%RC
Since
$\alga/\rad$
is semisimple, we have that
$\rad(\alga/\bim)\subseteq \rad/\bim$.
But
$(\rad/\bim)^2=0$.
Thus
$\rad(\alga/\bim)=\rad/\bim$.
Observe that
$\alga/\bim \in\clasekmj$
and
$\dim\alga/\bim\leq\dim\alga$.
Therefore there exists a subsuperalgebra
$\overline{\algs}\subseteq\alga/\bim $
such that
$\overline{\algs}\cong\frac{\alga/\bim}{\rad/\bim}\cong\alga/\rad$
and
$\alga/\bim=\overline{\algs}\oplus\rad/\bim$.
By the main theorems on homomorphisms, there is a subsuperalgebra
$\algs\subseteq \alga$
such that
$\bim\subseteq \algs$
and
$\algs/\bim\cong \overline{\algs}\cong\alga/\rad\cong\jor$.
We observe that
$\algs\in\clasekmj$
and
$\rad(\algs)=\bim$ is an irreducible $\jor$-superbimodule.
By the assumption, WPT is true for $\algs$,
hence there is a subsuperalgebra
$\algs_1\subseteq\algs\subseteq\alga$,
such that
$\algs_1\cong\algs/\bim\cong\alga/\rad$.
Since
$\algs_1$
is semisimple,
$\rad\cap\algs\subseteq\rad(\algs_1)=0$.
Furthermore,
$\dim\algs_1=\dim\alga-\dim\rad$.
Hence,
$\dim(\rad+\algs_1)=\dim\alga$
and
$\alga=\rad\oplus\algs_1 $.
\end{proof}

 Let
  $V_1,\ldots,V_k$  be  irreducible 
 $\jor$-superbimodules, and $\jor$ be a simple Jordan superalgebra.
  Let
 $\mathfrak{M}(\jor;V_1,\ldots,V_k)=\{V\,/\, V$
 is a
 $\jor$-superbimodule, 
 doesn't containing amoung its irreducible subsuperbimodule any copy isomorphic to one of the superbimodules
 $V_1,\ldots,V_k\}$
 %RC
 It is clear that
 $\mathfrak{M}$
 is closed with respect to taking of subsuperbimodules and homomorphic images.
 Thus it satisfies the conditions of Theorem \ref{teoredu}.

%RC
%Note that if
% $\alga$
% is a finite-dimensional Jordan superalgebra and
% $\rad$
%is its solvable radical, then
%$(\alga/\rad)_0\cong\alga_0/\rad_0$
%is a Jordan algebra.
%Therefore, there is a subalgebra
%$\algs\subseteq\alga_0$
%such that
%$\algs\cong(\alga/\rad)_0$
%and
%$\alga_0=\algs\oplus\rad_0$.
%Thus we need to verify whether there is an
%$\alga_0$-subsuperbimodule
%$\bim\subseteq\alga_1$
%such that
%$\bim\cong(\alga/\rad)_1\cong\alga_1/\rad_1$
%and
%$\alga_1=\bim\oplus\rad_1$.
%
%In view of Proposition \ref{pro2}, and Theorem \ref{teoredu},
%in each case we may assume that
%$\alga/\rad$
%is a simple Jordan  superalgebra and
%$\rad$
%is an irreducible $\alga/\rad$-superbimodule.

 In each section, we assume that 
 $\alga$ is a finite dimensional Jordan superalgebra over 
 $\corf$, 
 with radical $\rad$ 
 and such that $\radd=0$,  $\alga/\rad\cong\jor$, 
 where $\jor$ is a simple Jordan superalgebra and $\rad$ is an irreducible $\jor$-superbimodule. 
Moreover,  if $b_1,b_2,\ldots, b_n$ is an additive base of $\jor_0$, then 
 we assume that 
 $\wt{b}_1,\wt{b}_2,\ldots,\wt{b}_n$ is an additive base of $\alga_0$, moreover, $\wt{b}_i\cdot\wt{b}_j=\widetilde{b_ib_j}$.
If  $\alga_1/\rad_1\cong \jor_1$ and $v_1,\ldots,v_k$ is an additive base of $\jor_1$, we can assume that 
$\bav_1,\ldots,\bav_k$ is an additive base of $\alga_1/\rad_1$, 
and we shall find 
$\wv_1,\ldots,\wv_k$ additive base of $\alga_1$
such that $\wv_i\cdot \wv_j=\widetilde{v_iv_j}$ and $\wv_i\cdot\wt{b}_j=\widetilde{v_ib_j}$.
 In each case we can assume that 
 $\wt{a}\cdot n=an$, where $\wt{a}\in \alga_0\dot\cup\alga_1$, $a\in\jor_0\dot\cup\jor_1$, $n\in\rad_0\dot\cup\rad_1$

%RC:

\section{Kac superalgebra}\label{seckac}
 In this section, we consider the 10-dimensional Kac superalgebra 
 $\mathcal{K}_{10}=\jor_0\dotplus\jor_1$, where 
$\displaystyle{
\jor_0=(\corfd e+\sum_{i=1}^{4}\corfd v_i)\oplus \corfd f}$, 
$\displaystyle{\jor_1=\corfd x_1+\corfd x_2+\corfd y_1+\corfd y_2}$, 
and all nonzero products of the basis elements are the following
\begin{eqnarray}
\begin{aligned}\label{eqkac1}
\begin{array}{lllll}
e^2=e, & e\cdot v_i=v_i, & f^2=f,& v_1\cdot v_2=v_3\cdot v_4=2e.
 \end{array}
  \end{aligned} \\
 \begin{aligned}\label{parimparkac}
\begin{array}{lllll}
f\cdot x_j=\frac{1}{2}x_j, & f\cdot y_j=\frac{1}{2}y_j, & e\cdot x_j=\frac{1}{2}x_j, & e\cdot y_j=\frac{1}{2}y_j,\\ % j=1,2;\\
y_1\cdot v_1=x_2,& y_2\cdot v_1=-x_1,& x_1\cdot v_2=-y_2,& x_2\cdot v_2=y_1, \\
x_2\cdot v_3=x_1,& y_1\cdot v_3=y_2,& x_1\cdot v_4=x_2,& y_2\cdot v_4=y_1.
\end{array}
\end{aligned}\\
\begin{aligned}\label{eqkac2}
x_1\cdot x_2=v_1, &&  x_1\cdot y_2=v_3, && x_2\cdot y_1=v_4, && y_1\cdot y_2=v_2, \\ x_i\cdot  y_i=e-3f.
\end{aligned}
\end{eqnarray}

The zero characteristic of the ground field implies that 
$\mathcal{K}_{10}$ 
is a simple Jordan superalgebra. 
Consider the regular superbimodule over 
$\mathcal{K}_{10}$, 
$\reg(\kac)$ and assume that  
$a\leftrightarrow e$, $b\leftrightarrow f$,  $u_i\leftrightarrow v_i$, $m_j\leftrightarrow x_j$, 
  and 
  $n_j\leftrightarrow y_j$ 
 for $i=1,2,3,4,\, j=1,2$, thus
\begin{displaymath}\begin{array}{l}
(\reg K_{10})_0=\left(\corfd a+\corfd u_1+\corfd u_2+\corfd u_3+\corfd u_4\right)\oplus \corfd b, \\
(\reg K_{10})_1=\corfd m_1+\corfd m_2+\corfd n_1+\corfd n_2
\end{array}
\end{displaymath}
  Let $\alga_0=(\algs_0\oplus\rad_0)$ and $ \alga_1/\rad_1\cong(\kac)_1$. 
  Assume that 
  $\algs_0= \corfd \wt{e}+\sum_{i=1}^{4}\corfd \wt{v}_i\oplus \corfd \wt{f}$, and $(\kac)_0\cong\algs_0$
  and 
  $\alga_1/\rad_1= \corfd \bar{x}_1+\corfd\bar{x}_2+\corfd\bar{y}_1+\corfd\bar{y}_2$.

 \begin{lemma}\label{lemakaccero}
 	\begin{eqnarray}
 	 \begin{aligned}\label{kaclemma1parimpar}
 	\begin{array}{lllll}
 	\wt{f}\cdot \wt{x}_j=\frac{1}{2}\wt{x}_j, & \wt{f}\cdot \wt{y}_j=\frac{1}{2}\wt{y}_j, & \wt{e}\cdot \wt{x}_j=\frac{1}{2}\wt{x}_j, & 
 	\wt{e}\cdot \wt{y}_j=\frac{1}{2}\wt{y}_j,\\ % j=1,2;\\
 	\wt{y}_1\cdot \wt{v}_1=\wt{x}_2,& \wt{y}_2\cdot \wt{v}_1=-\wt{x}_1,& 
 	\wt{x}_1\cdot \wt{v}_2=-\wt{y}_2,& \wt{x}_2\cdot \wt{v}_2=\wt{y}_1, \\
 	\wt{x}_2\cdot \wt{v}_3=\wt{x}_1,& \wt{y}_1\cdot \wt{v}_3=\wt{y}_2,&
 	 \wt{x}_1\cdot \wt{v}_4=\wt{x}_2,& \wt{y}_2\cdot \wt{v}_4=\wt{y}_1,
 	\end{array}
 	\end{aligned}
 	\end{eqnarray}
 \end{lemma}
 \proof 
 Firts prove $\wt{e}\,\wt{x}_1=\frac{1}{2}\wt{x}_1$.  
To start, we can assume that  there exist scalars $\lambda^{x_1e}_{s_1}s_1$, such that 
$\wt{e}\,\wt{x}_1=\frac{1}{2}\wt{x}_1+\Lambda^{x_1}_e=\frac{1}{2}\wt{x}_1+\lambda^{x_1e}_{m_1}m_1+\lambda^{x_1e}_{m_2}m_2+\lambda^{x_1e}_{n_1}n_1+\lambda^{x_1e}_{n_2}n_2$.

It is easy to see that 
  $\Lambda^{x_i}_e\cdot \wt{e}=\frac{1}{2}\Lambda^{x_i}_e$. 
Substituting 
 $a_i=\wt{x}_1$ and $a_j=a_k=a_l=\wt{e}$ 
 in \eqref{idsupjordan}, we get
 $$2((\wt{x}_1\cdot \wt{e})\cdot \wt{e})\cdot \wt{e}+\wt{x}_1\cdot \wt{e}=3(\wt{x}_1\cdot \wt{e})\cdot \wt{e}.$$
 Combining the above equality with 
 $\wt{x}_1\cdot \wt{e}=\frac{1}{2}\wt{x}_1+\Lambda^{x_1}_e$, we have
 $\frac{5}{2}\Lambda^{x_1}_e=3\Lambda^{x_1}_e$, 
 therefore, 
 $\lambda^{x_1e}_{m_1}=\lambda^{x_1e}_{m_2}=\lambda^{x_1e}_{n_1}=\lambda^{x_1e}_{n_2}=0$. Thus, 
 $\Lambda^{x_1}_e=0$ and  $\wt{x}_1\cdot \wt{e}=\frac{1}{2}\wt{x}_1$.
 Similarly one can prove the equalities
 $\wt{x}_2\cdot \wt{e}=\frac{1}{2}\wt{x}_2$, $\wt{x}_i\cdot \wt{f}=\frac{1}{2}\wt{x}_i$, $\wt{y}_i\cdot \wt{e}=\frac{1}{2}\wt{y}_i$ and $\wt{y}_i\cdot \wt{f}=\frac{1}{2}\wt{y}_i$.

 \medskip 
 Now we shall prove that others equalities in \eqref{kaclemma1parimpar} hold.
 Let $\Lambda^{ij}_x$ be the radical part in the product $\wt{x}_i\cdot \wt{v}_j$ where $\Lambda^{ij}_x=\lambda^{ijx}_{m_1}m_1+\lambda^{ijx}_{m_2}m_2+\lambda^{ijx}_{n_1}n_1+\lambda^{ijx}_{n_2}n_2$ 
 for some scalars 
 $\lambda^{ijx}_{m_1}$, $\lambda^{ijx}_{m_2}$, $\lambda^{ijx}_{n_1}$ and $\lambda^{ijx}_{n_2}$. (Similarly, $\Lambda^{ij}_y$.)
 
 \medskip 
 Firts note that 
  $(\Lambda^{ij}_s\cdot \wt{v}_j)\cdot \wt{v}_j=0$ 
  for $s=x$ or $s=y$.
  
 \noindent 
 We set $a_i=\wt{y}_1$ and  $a_j=a_k=a_l=\wt{v}_1$ in \eqref{idsupjordan}. 
Since $\wt{v}_i^2=0$, we have  
  \begin{align*}
   0=((\wt{y}_1\cdot \wt{v}_1)\cdot \wt{v}_1)\cdot \wt{v}_1=((\wt{x}_2+\Lambda^{11}_y)\cdot \wt{v}_1)\cdot \wt{v}_1=(\wt{x}_2\cdot \wt{v}_1)\cdot \wt{v}_1=\Lambda^{21}_x\cdot \wt{v}_1.
  \end{align*}
 Thus, $\lambda^{21x}_{n_1}m_2-\lambda^{21x}_{n_2}m_1=0$. 
 The linear independence of 
 $m_1$ and $m_2$ implies 
 $\lambda^{21x}_{n_1}=\lambda^{21x}_{n_2}=0$ 
 and therefore $\Lambda^{21}_x=\lambda^{21x}_{m_1}m_1+\lambda^{21x}_{m_2}m_2$.
 Similarly one can prove that 
 $\Lambda^{11}_x=\lambda^{11x}_{m_1}m_1+\lambda^{11x}_{m_2}m_2$, $\Lambda^{12}_y=\lambda^{12y}_{n_1}n_1+\lambda^{12y}_{n_2}n_2$, $\Lambda^{22}_y=\lambda^{22y}_{n_1}n_1+\lambda^{22y}_{n_2}n_2$, $\Lambda^{13}_x=\lambda^{13x}_{m_2}m_2+\lambda^{13x}_{n_1}n_1$, $\Lambda^{23}_y=\lambda^{23y}_{m_2}m_2+\lambda^{23y}_{n_1}n_1$, $\Lambda^{24}_x=\lambda^{24x}_{m_1}m_1+\lambda^{24x}_{n_2}n_2$ and 
 $\Lambda^{14}_y=\lambda^{14y}_{m_1}m_1+\lambda^{14y}_{n_2}n_2$. 
 
 \medskip
 Substituing
 $a_i=\wt{y}_1$, $a_j=\wt{v}_1$ and $a_k=a_l=\wt{v}_2$ 
 in \eqref{idsupjordan}, we have, 
 \begin{equation}\label{q1122}
  ((\wt{y}_1\cdot \wt{v}_1)\cdot \wt{v}_2)\cdot \wt{v}_2+((\wt{y}_1\cdot \wt{v}_2)\cdot \wt{v}_2)\cdot \wt{v}_1+((\wt{v}_1\cdot \wt{v}_2)\cdot \wt{v}_2)\cdot \wt{y}_1=2(\wt{v}_1\cdot \wt{v}_2)\cdot (\wt{y}_1\cdot \wt{v}_2)
 \end{equation}
 Observe that 
 $\Lambda^{12}_y=\lambda^{12y}_{n_1}n_1+\lambda^{12y}_{n_2}n_2$, 
 therefore 
 $\Lambda^{12}_y\cdot \wt{v}_2=0$. 
 Recall that 
 $\wt{v}_1\cdot \wt{v}_2=2\wt{e}$, $\wt{e}\cdot \wt{v}_i=\wt{v}_i$ and $\wt{e}\cdot (\wt{y}_1\cdot \wt{v}_2)=\frac{1}{2}\wt{y}_1\cdot \wt{v}_2$. 
 Thus,  
 combining the above observation with \eqref{q1122}, we obtain the equality 
 \begin{align*}
 0
 =& ((\wt{y}_1\cdot \wt{v}_1)\cdot \wt{v}_2)\cdot \wt{v}_2
 =\Lambda^{12}_y+\Lambda^{22}_x\cdot \wt{v}_2 
 = \Lambda^{12}_y+\Lambda^{22}_x\cdot \wt{v}_2,
 \end{align*} therefore
 $( \lambda^{12y}_{n_1}+\lambda^{22x}_{m_2}) \,n_1+(\lambda^{12y}_{n_2}-\lambda^{22x}_{m_1})\,n_2=0$.
 Using the fact that $n_1$ and $n_2$ are linearly independent, we obtain
 $\lambda^{12y}_{n_1}=-\lambda^{22x}_{m_2}$ and  $\lambda^{12y}_{n_2}=\lambda^{22x}_{m_1}$. 
 
 Taking 
 $a_i=\wt{y}_1$, $a_j=\wt{v}_1$ and $a_k=a_l=\wt{v}_3$ 
 in \eqref{idsupjordan} we have, 
 $$0=((\wt{y}_1\cdot \wt{v}_1)\cdot \wt{v}_3)\cdot \wt{v}_3+((\wt{y}_1\cdot \wt{v}_3)\cdot \wt{v}_3)\cdot \wt{v}_1.$$
 Thus we obtain $\lambda^{13x}_{n_2}=-\lambda^{23x}_{n_1}$.

  Using 
  $0=((\wt{y}_2\cdot \wt{v}_1)\cdot \wt{v}_2)\cdot \wt{v}_2+((\wt{y}_2\cdot \wt{v}_2)\cdot \wt{v}_2)\cdot \wt{v}_1$, 
  we obtain 
  $\lambda^{12x}_{m_1}=\lambda^{22y}_{n_2}$ and $\lambda^{12x}_{m_2}=-\lambda^{22y}_{n_1}$.
  Since, 
  $0=((\wt{y}_2\cdot \wt{v}_1)\cdot \wt{v}_4)\cdot \wt{v}_4+((\wt{y}_2\cdot \wt{v}_4)\cdot \wt{v}_4)\cdot \wt{v}_1$,
  then 
  $\lambda^{24x}_{n_1}=-\lambda^{14x}_{n_2}$.
  
  \medskip
  Similarly, we obtain   
  $\lambda^{21x}_{m_1}=\lambda^{11y}_{n_2}$,
  $\lambda^{21x}_{m_2}=-\lambda^{11y}_{n_1}$,
  $\lambda^{23y}_{m_1}=-\lambda^{13y}_{m_2}$,
 $\lambda^{23y}_{n_2}=-\lambda^{13y}_{n_1}$,
 $\lambda^{23x}_{m_2}=-\lambda^{13x}_{m_1}$,
 $\lambda^{12y}_{n_1}=-\lambda^{12x}_{m_2}$,
 $\lambda^{12y}_{n_2}=\lambda^{12x}_{m_1}$,
 $\lambda^{24x}_{m_2}=-\lambda^{14x}_{m_1}$,
 $\lambda^{11x}_{m_2}=\lambda^{21y}_{n_1}$,
 $\lambda^{11x}_{m_1}=-\lambda^{21y}_{n_2}$,
 $\lambda^{14y}_{m_2}=-\lambda^{24y}_{m_1}$,
 $\lambda^{14y}_{n_1}=-\lambda^{24y}_{n_2}$,  
 $\lambda^{22y}_{n_2}=-\lambda^{12x}_{m_1}$,
 $\lambda^{13x}_{m_1}=-\lambda^{23x}_{m_2}$, 
 $\lambda^{11y}_{n_2}=\lambda^{11y}_{n_1}=0$.
 Thus, we have 
 $\Lambda^{21}_x=0$, $\Lambda^{12y}=\Lambda^{22y}=\lambda^{12y}_{n_1}n_1$ 
 and 
  $\Lambda^{11}_y=\lambda^{11y}_{m_1}m_1+\lambda^{11y}_{m_2}m_2$.
 
 Let $a_i=\wt{y}_1$, $a_j=a_l=\wt{v}_1$ and $a_k=\wt{v}_2$ in 
 \eqref{idsupjordan}. 
 Then we have 
 $\wt{y}_1\cdot \wt{v}_1=((\wt{y}_1\cdot \wt{v}_1)\cdot \wt{v}_2)\cdot \wt{v}_1$,
 therefore, 
 $\lambda^{22x}_{n_1}=-\lambda^{11y}_{m_2}$ and 
 $\lambda^{22x}_{n_2}=\lambda^{11y}_{m_1}$. 
 Similarly, we can obtain 
 $\lambda^{11x}_{m_1}=\lambda^{23x}_{n_2}$, 
 $\lambda^{11x}_{m_2}=-\lambda^{23x}_{n_1}$, 
 $\lambda^{24x}_{n_1}=\lambda^{14x}_{n_2}=0$, 
 $\lambda^{12x}_{n_1}=\lambda^{12x}_{n_2}=0$, 
 $\lambda^{13x}_{m_1}=-\lambda^{21y}_{n_1}$, 
 $\lambda^{13x}_{n_2}=0$, 
 $\lambda^{14x}_{n_1}=\lambda^{21y}_{n_2}$, 
 $\lambda^{21y}_{m_1}=\lambda^{21y}_{m_2}=0$, 
 $\lambda^{23y}_{n_1}=\lambda^{12x}_{m_2}$, 
 $\lambda^{24y}_{m_1}=0$, 
 $\lambda^{12y}_{n_1}=-\lambda^{24y}_{m_2}$, 
 $\lambda^{22y}_{n_1}=-\lambda^{13x}_{m_2}$,
 $\lambda^{13y}_{m_1}=-\lambda^{22x}_{m_2}$,  
 $\lambda^{23y}_{n_2}=0$,
 $\lambda^{14x}_{m_2}=-\lambda^{23x}_{m_1}$,  
 $\lambda^{24y}_{n_1}=-\lambda^{13y}_{n_2}$,  
 $\lambda^{14y}_{n_1}=-\lambda^{22x}_{n_2}$,  
 $\lambda^{11y}_{m_1}=\lambda^{24y}_{n_2}$. 
 
 Thus, we have 
 $\Lambda^{12}_x=\Lambda^{13}_x=\Lambda^{24}_x=\Lambda^{12}_y=\Lambda^{22}_y=\Lambda^{23}_y=0$.
 
 \medskip
 Setting 
 $a_i=\wt{y}_1$, $a_j=\wt{v}_1$, $a_k=\wt{v}_3$ and $a_l=\wt{v}_2$ 
 in  \eqref{idsupjordan}, we obtain 
 $$((\wt{y}_1\cdot \wt{v}_1)\cdot \wt{v}_3)\cdot \wt{v}_2+((\wt{y}_1\cdot \wt{v}_2)\cdot \wt{v}_3)\cdot \wt{v}_1+\wt{y}_1\cdot \wt{v}_3=0.$$
 Therefore we have 
 $\lambda^{11y}_{m_2}=\lambda^{13y}_{n_2}$.  
 Similarly one  can prove the equalities 
 $\lambda^{11x}_{m_1}=\lambda^{23x}_{m_1}=\lambda^{22x}_{n_1}=0$.
 Thus
 $\Lambda^{11}_{x}=\Lambda^{14}_{x}=\Lambda^{22}_{x}=\Lambda^{23}_{x}=
 \Lambda^{11}_{y}=\Lambda^{13}_{y}=\Lambda^{14}_{y}=\Lambda^{21}_{y}=\Lambda^{24}_{y}=0$. 
 \endproof

%\begin{equation}\label{kaclemma1par}
%	\begin{array}{lllll}
%	\wt{e}^2=\wt{e}, & \wt{e}\cdot \wt{v}_i=\wt{v}_i, & \wt{f}^2=\wt{f},& \wt{v}_1\cdot \wt{v}_2=\wt{v}_3\cdot \wt{v}_4=2\wt{e},
%	 \end{array}
%\end{equation}

\begin{lemma}\label{forkac}
There exist $\alpha\in\corf$ such that
\begin{equation}\begin{aligned}
\textnormal{\bf{(i)} } & \wt{x}_1\cdot  \wt{x}_2=\wt{u}_1+\alpha u_1, & 
\textnormal{\bf{(ii)} } & \wt{y}_1\cdot \wt{y}_2=\wt{y}_2+\alpha u_2 \\
\textnormal{\bf{(iii)} }&  \wt{x}_1\cdot \wt{y}_2=\wt{y}_3+\alpha u_3, & 
\textnormal{\bf{(iv)} } &  \wt{x}_2\cdot \wt{y}_1= \wt{y}_4+\alpha u_4\\
\textnormal{\bf{(v)} } &  \wt{x}_1\cdot \wt{y}_1=\wt{e}-3\wt{f}+\alpha a-3\alpha b &
\textnormal{\bf{(vi)} } &  \wt{x}_2\cdot \wt{y}_2=\wt{e}-3\wt{f}+\alpha a-3\alpha b 
\end{aligned}\end{equation}
\end{lemma}

\proof 
We can assume that there exist 
$\Lambda^{12}_{x}, \Lambda^{12}_{y}, \Lambda^{12}_{xy}, \Lambda^{21}_{xy}, \Lambda^{11}_{xy}$ and $\Lambda^{22}_{xy}\in\rad_0$ 
such that 
$\wt{x}_1\cdot \wt{x}_2=\wt{v}_1+\Lambda^{12}_x$,  
$\wt{y}_1\cdot \wt{y}_2=\wt{v}_2+\Lambda^{12}_x$,
$\wt{x}_1\cdot \wt{y}_2=\wt{v}_3+\Lambda^{12}_{xy}$,
$\wt{x}_2\cdot \wt{y}_1=\wt{v}_4+\Lambda^{21}_{xy}$,
$\wt{x}_1\cdot \wt{y}_1=\wt{e}-3\wt{f}+\Lambda^{11}_{xy}$ and 
$\wt{x}_2\cdot \wt{y}_2=\wt{e}-3\wt{f}+\Lambda^{22}_{xy}$.

\medskip 
We assume that there exist 
$\eta^{tij}_a$, $\eta^{tij}_b$, $\eta^{tij}_{u_1}$, $\eta^{tij}_{u_2}$, $\eta^{tij}_{u_3}$ 
and 
$\eta^{tij}_{u_4}\in\corf$ such that
 $\Lambda^{ij}_{t}=\eta^{tij}_aa+\eta^{tij}_bb+\eta^{tij}_{u_1}u_1+\eta^{tij}_{u_2}u_2+\eta^{tij}_{u_3}u_3+\eta^{tij}_{u_4}u_4$
for $i,j\in\{1,2\}$ and $t\in\{x,y,xy\}$.

\noindent 
Replacing  
$a_i=\wt{x}_1$, $a_j=\wt{x}_2$, $a_k=a_l=\wt{v}_1$ in the equation \eqref{idsupjordan} and using \eqref{kaclemma1parimpar}; 
we have 
$((\wt{x}_1\cdot \wt{x}_2)\cdot \wt{v}_1)\cdot \wt{v}_1=0$, thus 
\begin{equation*}
\begin{aligned}
0= & 
((\wt{v}_1+\eta^{x12}_aa+\eta^{x12}_bb+\eta^{x12}_{u_1}u_1+\eta^{x12}_{u_2}u_2+\eta^{x12}_{u_3}u_3+\eta^{x12}_{u_4}u_4)
\cdot \wt{v}_1)\cdot \wt{v}_1\\
=& (\eta^{x12}_a u_1+2\eta^{x12}_{u_2}a)\cdot \wt{v}_1=2\eta^{x12}_{u_2}u_1.
\end{aligned}
\end{equation*} 
Therefore, 
$\eta^{x12}_{u_2}=0$. 
In the same way one can prove that 
$\eta^{x12}_{u_4}=\eta^{x12}_{u_3}=0$, 
thus
\begin{equation}\label{p1p2}
\wt{x}_1\cdot \wt{x}_2=\wt{v}_1+\eta^{x12}_aa+\eta^{x12}_bb+\eta^{x12}_{u_1}u_1.
\end{equation}
Similarly,  we obtain that $\wt{y}_1\cdot \wt{y}_2=\wt{v}_2+\eta^{y12}_aa+\eta^{y12}_bb+\eta^{y12}_{u_2}u_2.$

Since \eqref{kaclemma1parimpar} and replacing 
$a_i=\wt{x}_1, \, a_j=\wt{x}_2$ and $a_k=a_t=\wt{v}_2$ 
in \eqref{idsupjordan},  we obtain
\begin{align}\label{igualdad13}
((\wt{x}_1\cdot \wt{x}_2)\cdot \wt{v}_2)\cdot \wt{v}_2=2(\wt{x}_1\cdot \wt{v}_2)\cdot( \wt{v}_2\cdot \wt{x}_2)=2\wt{y}_1\cdot \wt{y}_2.
\end{align}
Replacing  \eqref{p1p2} and its equivalent for $\wt{y}_1\cdot \wt{y}_2$ in \eqref{igualdad13}, 
we obtain
$2\wt{v}_2+\eta^{x12}_{u_1}u_2=2(\wt{v}_2+\eta^{y12}_aa+\eta^{y12}_bb+\eta^{y12}_{u_2}u_2)$, 
therefore 
$\eta^{y12}_a=\eta^{y12}_b=0$ 
and 
$\eta^{x12}_{u_1}=\eta^{y12}_{u_2}$. 
If we take 
$a_i=\wt{y}_1, \, a_j=\wt{y}_2$ and $a_k=a_t=\wt{v}_1$ in \eqref{idsupjordan} 
we obtain 
$\eta^{x12}_a=\eta^{x12}_b=0$.

Let 
$a_i=\wt{x}_1, \, a_j=\wt{y}_2$ and $a_k=a_t=\wt{v}_1$ 
in \eqref{idsupjordan},  thus, we obtain $\eta^{xy12}_{u_2}=0$.
If we shall take 
$a_k=a_t=\wt{v}_2$ or 
$a_k=a_t=\wt{v}_3$ we obtain 
$\eta^{xy12}_{u_1}=\eta^{xy12}_{u_4}=0$.
Similarly to above case  we obtain
$\eta^{xy21}_{u_1}=\eta^{xy21}_{u_2}=\eta^{xy21}_{u_3}=0$. 

Setting 
$a_i=\wt{v}_1, \, a_j=\wt{y}_2$ and $a_k=a_t=\wt{v}_4$ 
(respectively, $a_i=\wt{x}_2, \, a_j=\wt{y}_1$ and  $a_k=a_t=\wt{v}_3$)
in \eqref{idsupjordan},
we obtain $\eta^{xy21}_a=\eta^{xy21}_b=0$ 
(respectively $\eta^{xy12}_a=\eta^{xy12}_b=0$) and 
$\eta^{xy21}_{u_4}=\eta^{xy12}_{u_3}$.

If we take
$a_i=\wt{x}_1$,  $a_j=\wt{x}_2$,  $a_k=\wt{v}_2$ and $a_t=\wt{v}_4$ 
in \eqref{idsupjordan}, then using \eqref{kaclemma1parimpar} we have 
$((\wt{x}_1\cdot \wt{x}_2)\cdot \wt{v}_2)\cdot \wt{v}_4=2\wt{x}_2\cdot \wt{y}_1$.  
Therefore,
$\eta^{xy21}_{u_4}=\eta^{x12}_{u_1}$. 
Thus we get
$\wt{x}_1\cdot \wt{x}_2=\wt{v}_1+\alpha u_1$, $\wt{y}_1\cdot \wt{y}_2=\wt{v}_2+\alpha u_2$, $\wt{x}_1\cdot \wt{y}_2=\wt{v}_3+\alpha u_3$ and 
$\wt{x}_2\cdot \wt{y}_1= \wt{v}_4+\alpha u_4$ 
for some $\alpha\in\corf$.
 
\medskip 
Let  
 $a_i=\wt{x}_1$, $a_j=\wt{y}_1$, and $a_k=a_t=\wt{v}_1$ 
 in \eqref{idsupjordan}. 
 Using the products in  $\algs_0$ and \eqref{kaclemma1parimpar}, 
 we obtain 
 $((\wt{x}_1\cdot \wt{y}_1)\cdot \wt{v}_1)\cdot \wt{v}_1=0$.
 Thus 
 $((\wt{x}_1\cdot \wt{y}_1)\cdot \wt{v}_1) \cdot \wt{v}_1=0$ 
 and therefore 
$\eta^{xy11}_{u_2}=0$. 
Analogously, one can verify that  
$\eta^{xy11}_{u_1}=\eta^{xy11}_{u_3}=\eta^{xy11}_{u_4}=0$.
Thus 
$\wt{x}_1\cdot \wt{y}_1=\wt{e}-3\wt{f}+\eta^{xy11}_{a}a+\eta^{xy11}_{b}$.
Similarly one can show that  
$\wt{x}_2\cdot \wt{y}_2=\wt{e}-3\wt{f}+\eta^{xy22}_{a}a+\eta^{xy22}_{b}$.

\medskip
Taking
$a_i=\wt{x}_1,\,a_j=\wt{y}_1,\, a_k=\wt{v}_1$ and $a_t=\wt{v}_2$, 
in \eqref{idsupjordan},  we obtain
\begin{equation}
((\wt{x}_1\cdot  \wt{y}_1)\cdot \wt{v}_1)\cdot \wt{v}_2+\wt{x}_1\cdot \wt{y}_1= 
2\wt{e}\cdot (\wt{x}_1\cdot \wt{y}_1)+\wt{x}_2\cdot \wt{y}_2.
\end{equation}
From the above equality, it is easy to see that 
$\eta^{xy11}_a=\eta^{xy22}_a$ and 
$\eta^{xy11}_b=\eta^{xy22}_b$. 
Thus we have that 
$\wt{x}_1\cdot \wt{y}_1=\wt{x}_2\cdot \wt{y}_2$.

Let 
 $a_i=\wt{x}_1$, $a_j=\wt{x}_2$ and   $a_k=a_t=wt{y}_1$ in (\ref{idsupjordan}), 
hence
 \begin{equation*}\label{igualdad14}
 \begin{aligned}
0=& ((\wt{x}_1\cdot \wt{x}_2)\cdot \wt{y}_1)\cdot \wt{y}_1-((\wt{x}_1\cdot \wt{y}_1)\cdot \wt{y}_1)\cdot \wt{x}_2+
((\wt{x}_2\cdot \wt{y}_1)\cdot \wt{y}_1)\cdot \wt{x}_1\\
=& ((\wt{v}_1+\alpha u_1)\cdot \wt{y}_1)\cdot \wt{y}_1-((\wt{e}-3\wt{f}+\eta^{xy11}_aa+\eta^{xy11}_bb)\cdot \wt{y}_1)\cdot \wt{x}_2 +
((\wt{v}_4+\alpha u_4)\cdot \wt{y}_1)\cdot \wt{y}_1 \\
=&( \wt{y}_1\cdot \wt{v}_1+\alpha \wt{y}_1\cdot u_1)\cdot \wt{y}_1-(-\wt{y}_1+\frac{1}{2}(\eta^{xy11}_a+\eta^{xy11}_b)n_1)\cdot \wt{x}_2 \\
=& \wt{x}_2\cdot \wt{y}_1+\alpha m_2\cdot \wt{y}_1-\wt{x}_2\cdot \wt{y}_1+\frac{1}{2}(\eta^{xy11}_a+\eta^{xy11}_b)\wt{x}_2\cdot n_1
= (\alpha +\frac{1}{2}(\eta^{xy11}_a+\eta^{pq11}_b))u_4, 
\end{aligned}
 \end{equation*}
thus,
$2\alpha=-(\eta^{xy11}_a+\eta^{xy11}_b)$.

\medskip
If we take 
$a_i=\wt{x}_1$, $a_j=\wt{y}_1$, $a_k=\wt{x}_2$ and $a_t=\wt{y}_2$ in \eqref{idsupjordan}, 
we obtain 
$\eta^{xy11}_a=\alpha$ and $\eta^{xy11}_b=-3\alpha$, 
therefore
$\wt{x}_1\cdot \wt{y}_1=\wt{x}_2\cdot \wt{y}_2=\wt{e}-3\wt{f}+\alpha a-3\alpha b$. 
 \endproof

 \begin{lemma}\label{lemasolkac} 
There exists $\beta\in\corf$ such that
$ \wt{x}_i=x_i+\beta\, m_i,$ and $\wt{y}_i=y_i+\beta\, n_i.$
for $i=1,2$.
 \end{lemma}

 \proof 
 Assume  that there exist 
 $\Lambda^{i}_x$ and $\Lambda^{i}_y\in\rad_1$ 
 such that 
 $\wt{x}_i=x_i+\Lambda^{i}_x$ and $\wt{y}_i=y_i+\Lambda^{i}_y$, 
 where 
 $\Lambda^{i}_t=\lambda^{t_i}_{m_1}m_1+\lambda^{t_i}_{m_2}m_2+\lambda^{t_i}_{n_1}n_1+\lambda^{t_i}_{n_2}n_2$
 and  
 $\lambda^{t_i}_{m_1}$, $\lambda^{t_i}_{m_2}$, $\lambda^{t_i}_{n_1}$ and $\lambda^{t_i}_{n_2}\in\corf$. It is easy to see that 
 $\wt{x}_i=x_i+\lambda^{x_i}_{m_i}\,m_i$
   and 
   $\wt{y}_i=y_i+\lambda^{y_i}_{n_i}\,n_i$.
 
\noindent 
Using the Lemma \ref{lemakaccero}, we have that 
 $\wt{x}_1\cdot \wt{v}_2=-\wt{y}_2$, 
 $\wt{x}_1\cdot \wt{v}_4=-\wt{x}_2$ and 
 $\wt{y}_2\cdot \wt{v}_4=\wt{y}_1$.
Thus one easily verifies that
  $\lambda^{x_1}_{m_1}=\lambda^{y_2}_{n_2}$, 
  $\lambda^{x_1}_{m_1}=\lambda^{x_2}_{m_2}$
  and  
  $\lambda^{y_2}_{n_2}=\lambda^{y_1}_{n_1}$.
  Therefore,
  $\lambda^{y_1}_{n_1}=\lambda^{y_2}_{n_2}=\lambda^{x_1}_{m_1}=\lambda^{x_2}_{m_2}.$ 
 \endproof

Let us prove the following theorem
\begin{theorem}\label{casokac}
Let 
$\alga$ 
be a finite dimensional Jordan superalgebra with solvable radical 
$\rad$ such that $\radd=0$ 
and 
$\alga/\rad\cong\kac$. 
Then there exists a subsuperalgebra 
$\algs\subseteq\alga$ 
such that 
$\algs\cong\kac$ and $\alga=\algs\oplus\rad$.
\end{theorem}
\proof
Recall that A. S. Shtern \cite{Sht1} 
proved that any irreducible Jordan superbimodule over $\kac$ is isomorphic to $\reg(\kac)$.
ByTheorem \ref{teoredu}, 
we only need to consider this case.
 
By Lemma \ref{forkac}, 
we can assume that there exists 
$\alpha\in\rad$ such that 
\begin{equation}\label{kacecuacionc}\begin{aligned}
\textnormal{\bf{(i)} } & \wt{x}_1\cdot  \wt{x}_2=\wt{u}_1+\alpha u_1, & 
\textnormal{\bf{(ii)} } & \wt{y}_1\cdot \wt{y}_2=\wt{y}_2+\alpha u_2 \\
\textnormal{\bf{(iii)} }&  \wt{x}_1\cdot \wt{y}_2=\wt{y}_3+\alpha u_3, & 
\textnormal{\bf{(iv)} } &  \wt{x}_2\cdot \wt{y}_1= \wt{y}_4+\alpha u_4\\
\textnormal{\bf{(v)} } &  \wt{x}_1\cdot \wt{y}_1=\wt{e}-3\wt{f}+\alpha a-3\alpha b &
\textnormal{\bf{(vi)} } &  \wt{x}_2\cdot \wt{y}_2=\wt{e}-3\wt{f}+\alpha a-3\alpha b 
\end{aligned}\end{equation}

By Lemma \eqref{lemasolkac}, there is a
$\beta\in\corf$ such that 
 $\wt{x}_i=x_i+\beta\, m_i$ and 
 $\wt{y}_i=y_i+\beta\, n_i$.

It is easy to verifies the following equalities
\begin{equation}\label{kacecua1}
 \begin{aligned}
 \wt{x}_1\cdot \wt{y}_1  =  x_1\cdot y_1+2\beta (a-3b), & & 
 \wt{x}_2\cdot \wt{y}_2  =  x_2\cdot y_2+2\beta(a-3b), \\
 \wt{x}_1\cdot \wt{x}_2  = x_1\cdot x_2+2\beta  u_1,&  & 
 \wt{y}_1\cdot \wt{y}_2  = y_1\cdot y_2+2\beta  u_2,\\
 \wt{x}_1\cdot \wt{y}_2  =  x_1\cdot y_2+2\beta u_3,& & 
 \wt{x}_2\cdot \wt{y}_1  = x_2\cdot y_2+2\beta  u_3.
 \end{aligned}
 \end{equation}

Using \eqref{kacecuacionc} and \eqref{kacecua1},  
we get 
$ \wt{x}_i\cdot \wt{y}_i =\wt{e}-3\wt{f}$, $\wt{x}_1\cdot\wt{x}_2=\wt{v}_1$,  
$\wt{x}_1\cdot\wt{y}_2=\wt{v}_3$, $\wt{x}_2\cdot\wt{y}_1=\wt{v}_4$ 
and  
$\wt{y}_1\cdot\wt{y}_2=\wt{v}_2$ if and only if, $2\beta=\alpha$.
This equality has always a solution. Therefore the WPT holds in the case under consideration.
\endproof

%%%%%%%%%%%%%%%%%%%%%%%%%%%%%%%%%%%%%%%%%%%%%%%%%%%%%%%%%%%%%%%%%%%%%%%%%%%%%%%%%%%%%%%%%%%%%%%%%%%%%%%%%%%%%%%%%%%%%%%%%%%%%%%%%%%%%%%%%%%%%%%%%%%%%%%%%%%%%%%%%%%%%%%%%%%%%%%%%%%%%%%%%%%%%%%%%%%%%%%%%%%%%%%%%%%%%%%%%%%%%%%%%%%%%%%%%%%%%%%%%%%%%%%%%%%%%%%%%%%%%%%%%%%%%%%%%%%%%%%%%%%%%%%%%%%%%%%%%%%%%%%%%%%%%%%%%%%%%%%%%%%%%%%%%7777777777777777777777777777777777777777777777777777777777777777777777777777777777

%%%%%%%%%%%%%%%%%%%%%%%%%%%%%%%%%%%%%%%%%%%%%%%%%%%%%%%%%%%%%%%%%%%%%%%%%%%%%%%
%%%%%%%%%%%%%%%%%%%%%%%%%%%%%%%%%%%%%%%%%%%%%%%%%%%%%%%%%%%%%%%%%%%%%%%%%%%%%%%
%%%%%%%%%%%%%%%%%%%%%%%%%%%%%%%%%%%%%%%%%%%%%%%%%%%%%%%%%%%%%%%%%%%%%%%%%%%%%%%22222222222222222222222222222222222222222222222222222222222222222222222222222222222222222222222222222222222222222
\section{Jordan superalgebra of superform.}
%RC
In this section we use the classification of irreducible $\jor$-bimodules obtained by E. Zelmanov and C. Martinez in \cite{zelmar2}, where 
$\jor=\jor(\mathcal{V},f)=(\corfd 1\oplus \mathcal{V}_0)\dotplus \mathcal{V}_1$
be a Jordan superalgebra of  nondegenerate super-symmetric superform
$f$
on a superspace
$\mathcal{V}$.

\medskip
We may assume that
$\dim \mathcal{V}_1>1$.
Let
$v_1,\ldots,v_n$
be an
$f$-orthonormal basis of $\mathcal{V}_0$,
i.e.
$f(v_i,v_i)=1, f(v_i,v_j)=0$ for $i\neq j, \quad i,j=1\ldots,n$.
Let
$w_1,\ldots,w_{2m}$
be a basis of
$\mathcal{V}_1$
such that
$f(w_{2p-1}, w_{2p})=1, \,1\leq p\leq m$,
and all the other products of basis elements are zero.

%RC
We know that all products
$v_1^{i_1}\cdots v_n^{i_n}w_1^{k_1}\cdots w_{2m}^{k_{2m}}$
form a basis of
$\algc$,
where
$ i_1,\ldots$, $i_n \in\{0,1\} $
and
$k_1, \ldots, k_{2m}$
are nonnegative integers and
$\algc$
denotes the Clifford superalgebra of
$\mathcal{V}$.
Let
$\algc_r$
be the subspace in
$\algc$
spanned by the  products of basis elements 
of length at most
$r$,
and let
$\jor=(\corfd 1 +\mathcal{V}_0)\dotplus \mathcal{V}_1$
be the Jordan superalgebra of superform
$f$.
Let
$a$
be an even vector,
$\mathcal{V}^{\prime}=\mathcal{V}\oplus\corfd a$.
We extend the superform
$f$
to
$\mathcal{V}^{\prime}$
so that
$f(a,a)=1, \, f(a,\mathcal{V})=0$.
Denote by 
$\algc_r^{\prime}$
the subspace in
$\algc^{\prime}$ defined in the same way as $\algc_r$ in $\algc$.

%RC
In this section, for every element
$v_1^{i_1}\cdots v_n^{i_n}w_1^{k_1}\cdots w_{2m}^{k_{2m}}$
of the basis of
$\algc$,
we put into correspondence 
a pair
$(I,K)$,
where
$I=(i_1,\ldots,i_n)$
is a
$n$-tuple and
$K=(k_1,\ldots,k_{2m})$
is a
$2m$-tuple where
$i_s,\, k_t$
satisfies the above conditions.
We write
$\eta_{I,K}=v_1^{i_1}\cdots v_n^{i_n}w_1^{k_1}\cdots w_{2m}^{k_{2m}}=V_IW_K$.
Note that for any pair of elements
$\eta_{I,K}, \eta_{I^\prime, K^\prime}\in \algc$,
the following relation holds
$\eta_{I,K}=\eta_{I^\prime, K^\prime}$ if and only if 
$I=I^\prime, \textnormal{ and } K=K^\prime.$
Thus every  element of the basis of
$\algc$
has a unique representation in terms of
$(I,K)$.
We denote
$V_{(0)}=1,\,V_{(1)}=v_1v_2\cdots v_n$.

%RC
Let
$\mathcal{I},$ $ \mathcal{K}$
be the following sets
\begin{align*}
\mathcal{I}&=\{I=(i_1,\ldots,i_n),  i_j=0 \textnormal{ or } 1,\, j=1,\ldots,n\,\},\\
\mathcal{K}&=\{{K}=(k_1,\ldots,k_{2m}),  k_j\in\mathbb{Z}^+\cup\{0\},\, j=1,\ldots,2m\,\},
\end{align*}
For
${I}\in \mathcal {I}$,
${K}\in \mathcal{K}$,
we denote
$|{I}|=i_1+\cdots+i_n$,
$|{K}|=k_1+\cdots+k_{2m}$ and
$|\eta_{I,K}|=|I|+|K|$.

\medskip
%RC
\noindent\textbf{Some relations in $\algc\am$.}

\medskip
%RC
From symmetric product in superalgebra $\algc^{+}$ we have 
\begin{align}
V_IW_K\circ v_j = \Big(-\frac{1}{2}\Big)^{i_1+\cdots +i_{j-1}}V_{(i_1,\ldots,i_j+1,\ldots, i_n)}W_{K}(1+(-1)^{|\eta_{I,K}|-i_j}),\label{leymulgen} 
\end{align}
\begin{equation}\label{generalwimpar}
\begin{aligned}
V_IW_{K}\circ w_p =\frac{1}{2}
V_IW_{(k_1,\ldots,k_{p}+1,\ldots,k_{2m})}(1&+(-1)^{|\eta_{I,K}|}) -\\ & k_{p+1}V_IW_{(k_1,\ldots,k_{p+1}-1,\ldots,k_{2m})}
\end{aligned}
\end{equation}
for $j=1,\ldots, n$ and 
$p=1,3,\ldots, 2m-1$.
We  note that a similar relation to  \eqref{generalwimpar} with some change of signs holds for even
$p$.

In this section,  $\alga_0=(\algs_0\oplus\rad_0)$ and $(\algs_1/\rad_1)\cong\jor_1$. Assume that  
  $\algs_0= \corfd1+\corfd \wt{v}_1+\cdots+ \corfd \wt{v}_n$, $\jor_0\cong\algs_0$ and 
  $\alga_1/\rad_1= \corfd \bar{w}_1+\corfd\bar{w}_2+\cdots+\corfd\bar{w}_{2m}$. 
We consider two cases for $\rad$.

\subsection{$\rad$ is isomorphic to $\algc_r/\algc_{r-2}$.}

Without loss of generality, we can take  
\begin{align*}
\rad_0&= \textnormal{{\bf{\vect}}}_\corf\,\langle\, \eta_{I,K}, |\eta_{I,K}|=r,r-1  \textnormal{ and } |K| \textnormal{ is even } \rangle,  \\ 
\rad_1&= \textnormal{{\bf{\vect}}}_\corf\,\langle\,  \eta_{I,K}, |\eta_{I,K}|=r,r-1 \textnormal{ and } |K| \textnormal{ is odd } \rangle.
 \end{align*}
Using the notation introduced above,  due to the equations 
\eqref{leymulgen} and \eqref{generalwimpar},  
we have the following products:
\begin{align}\label{accionpar}
\eta_{I,K}\cdot \wv_j=\left\{
\begin{array}{lll}
\pm V_{(i_1,\ldots,i_{j-1},0,i_{j+1}, \ldots,i_n)}W_K %v_1^{i_1}\cdots v_{j-1}^{i_{j-1}}v_{j+1}^{i_{j+1}}\cdots v_n^{i_n} W_K 
& \textnormal { if }  & |\eta_{I,K}|=r,\, i_j=1,\\
\pm V_{(i_1,\ldots,i_{j-1},1,i_{j+1}, \ldots,i_n)}W_K%v_1^{i_1}\cdots v_j\cdots  v_n^{i_n}W_K 
&\textnormal { if } 
& |\eta_{I,K}|=r-1,\, i_j=0.
\end{array}\right.\\
\label{accionimpar}
\eta_{I,K}\cdot \ww_p=\left\{
\begin{array}{lll}
\pm k_{p\pm 1}V_IW_{(k_1,\ldots,k_{p\pm 1}-1,\ldots,k_{2m})} %w_1^{k_1}\cdots w_{p\pm 1}^{k_{p\pm 1}-1}\cdots w_{2m}^{k_{2m}}
& \textnormal { if } & |\eta_{I,K}|=r, \\ 
V_IW_{(k_1,\ldots,k_{p}+1,\ldots,k_{2m})} %w_1^{k_1}\cdots w_p^{k_p+1}\cdots w_{2m}^{k_{2m}} 
&\textnormal { if }& |\eta_{I,K}|=r-1. \end{array}\right.
\end{align}

\noindent 
Firts, we prove three lemmas.
\begin{lemma}\label{Afirmacion1} 
$\wv_j\cdot \ww_s=0$.
\end{lemma}
\begin{proof} 
Setting $a_i=\ww_s$, $a_j=a_l=\wv_i$ and  $a_k=\wv_j$ in \eqref{idsupjordan}, we have
\begin{equation}\label{auxsupfo2}
((\ww_s\cdot \wv_i)\cdot \wv_j)\cdot \wv_i=0.
\end{equation}

we may assume that there exist some scalars 
$\xi_{(I,K)}^k$ 
such that
$$
\displaystyle{\wv_k\cdot \ww_s=\sum_{\substack{{I,\,K}\\{|\eta_{I,K}|=r }\\ |K| \textnormal{ odd}}}
\xi^{k}_{(I,K)}V_IW_K+\sum_{\substack{{I,\,K}\\{|\eta_{I,K}|=r-1}\\ |K| \textnormal{ odd}}}
\xi^{k}_{(I,K)}V_IW_K}.
$$

\medskip Let 
$\xi^{k}_{(I,k)}\eta_{I,K}\in\rad_1$ 
be a nonzero element and 
$j\neq k$, 
$j, k\in\{1,2,\ldots, n\}$.
Using \eqref{accionpar} and \eqref{auxsupfo2}, we obtain the following relations:

\noindent {\bf{(a)}} If $|\eta_{I,K}|=r-1$ and $i_j=1$, 
then
$(\xi_{(I,K)}^{k}\eta_{I,K}\cdot \wv_j)\cdot \wv_k=0$.\\
\noindent {\bf{(b)}} If $|\eta_{I,K}|=r-1$, $i_j=i_k=0$, 
then 
$(\xi_{(I,K)}^{k}\eta_{I,K}\cdot \wv_j)\cdot \wv_k=0$.  \\
\noindent {\bf{(c)}} If $|\eta_{I,K}|=r$, $i_j=0$,
then 
$(\xi_{(I,K)}^{k}\eta_{I,K}\cdot \wv_j)\cdot \wv_k=0$.  \\ 
\noindent {\bf{(d)}} If $|\eta_{I,K}|=r$, $i_j=i_k=1$, 
then
$(\xi_{(I,K)}^{k}\eta_{I,K}\cdot \wv_j)\cdot \wv_k=0$.  

\medskip From {\bf(a)} - {\bf{(d)}} and \eqref{auxsupfo2},  we have
\begin{align*}
\wv_k\cdot \ww_s=
\sum_{\substack{{I,\,K}\\{|\eta_{I,K}|=r-1}\\ |K| \textnormal{ odd}}} 
& 
\xi^{k}_{(I,K)} V_{(i_1,\ldots,i_{j-1},0,i_{j+1}, \ldots, i_n)}W_K %v_1^{i_1}\cdots v_{j-1}^{i_{j-1}}v_{j+1}^{i_{j+1}}\cdots v_n^{i_n}
+\sum_{\substack{{I,\,K}\\{|\eta_{I,K}|=r }\\ |K| \textnormal{ odd}}}
 \xi^{k}_{(I,K)}  V_{(i_1,\ldots,i_{j-1},1,i_{j+1}, \ldots, i_n)}W_K +%v_1^{i_1}\cdots v_j\cdots v_n^{i_n} W_K
 \nonumber\\
&
\sum_{\substack{{I,\,K}\\{|\eta_{I,K}|=r-1}\\ |K| \textnormal{ odd}}} 
\xi^{k}_{(I,K)} V_{(i_1,\ldots,i_{j-1},0,i_{j+1}, \ldots, i_{k-1},0,i_{k+1}, \ldots,i_n)}W_K +
%v_1^{i_1}\cdots v_{j-1}^{i_{j-1}}v_{j+1}^{i_{j+1}}\cdots   v_{k-1}^{i_{k-1}}v_{k+1}^{i_{k+1}}\cdots  v_n^{i_n}W_K
\\ 
&
\sum_{\substack{{I,\,K}\\{|\eta_{I,K}|=r }\\ |K| \textnormal{ odd}}} 
\xi^{k}_{(I,K)} V_{(i_1,\ldots,i_{j-1},0,i_{j+1}, \ldots, i_{k-1},0,i_{k+1}, \ldots,i_n)}W_K 
%v_1^{i_1}\cdots v_{j-1}^{i_{j-1}}v_{j+1}^{i_{j+1}}\cdots   v_{k-1}^{i_{k-1}}v_{k+1}^{i_{k+1}}\cdots v_n^{i_n}W_K
=0.
\end{align*} 
Thus,
\begin{align*}
\wv_k&\cdot \ww_s=
\sum_{\substack{{I,\,K}\\{|\eta_{I,K}|=r-1}\\ |K| \textnormal{ odd}}} 
\xi^{k}_{(I,K)}v_1^{i_1}\cdots v_j\cdots v_n^{i_n}W_K
+\sum_{\substack{{I,\,K}\\{|\eta_{I,K}|=r-1}\\ |K| \textnormal{ odd}}} \xi^{k}_{(I,K)}v_1^{i_1}\cdots v_{j}\cdots   v_{k}\cdots  v_n^{i_n}W_K+\nonumber\\
& \sum_{\substack{{I,\,K}\\{|\eta_{I,K}|=r }\\ |K| \textnormal{ odd}}} \xi^{k}_{(I,K)}v_1^{i_1}\cdots v_{j-1}^{i_{j-1}}v_{j+1}^{i_{j+1}}\cdots v_n^{i_n} W_K
+\sum_{\substack{{I,\,K}\\{|\eta_{I,K}|=r }\\ |K| \textnormal{ odd}}} \xi^{k}_{(I,K)}v_1^{i_1}\cdots v_{j}\cdots   v_{k}\cdots v_n^{i_n}W_K=0.
\end{align*}

\noindent
 If we apply \eqref{auxsupfo2} to the obtained above equation  for all $j\neq k$,
  we get
 \begin{equation}\label{ukzp0}
 \wv_k\cdot \ww_s= \sum_{\substack{{I,\,K}\\{|\eta_{I,K}|=r-1}\\ |K| \textnormal{ odd}}} \xi^{k}_{(I,K)}v_1\cdots  v_k^{i_{k}}\cdots v_nW_K+ \sum_{\substack{{K}\\{|K|=r }\\ |K| \textnormal{ odd}}} \xi^{k}_{(0,K)} W_K
 +\sum_{\substack{{K}\\{|K|=r-n }\\ |K| \textnormal{ odd}}} \xi^{k}_{(1,K)}V_{(1)}W_K.
 \end{equation}

\medskip Substituting
$a_i$ by $\ww_s$ and $a_j=a_k=a_l$ by $\wv_k$ respectively in \eqref{idsupjordan},  we obtain 
\begin{align}\label{auxzuuu}
 ((\ww_s\cdot \wv_k)\cdot \wv_k)\cdot \wv_k=z_s\cdot \wv_k.
\end{align}

Applying \eqref{auxzuuu}  to \eqref{ukzp0}, to get
\begin{align}\label{ukzp1}
\wv_k\cdot \ww_s=&\sum_{\substack{{I,\,K}\\{|I|+|K|=r-1}\\ |K| \textnormal{ odd}}} \xi^{k}_{(I,K)}v_1\cdots v_k^{i_{k}}\cdots v_nW_K.
 \end{align}
Substituting 
$a_i\,,a_j,\,a_k$, and $a_l$ by $\ww_s,\,\wv_k,\, \wv_j$ and $\wv_j$  
respectively in \eqref{idsupjordan},  
we have
\begin{align*}
((\ww_s\cdot\wv_k)\cdot \wv_j)\cdot\wv_j+((\ww_s\cdot\wv_j)\cdot\wv_j)\cdot \wv_k=\ww_s\cdot\wv_k
\end{align*}
 Applying the obtained equality to \eqref{ukzp1}, we have
$\wv_k\cdot \ww_s=0$. 
\end{proof}

\begin{lemma}\label{lemaformaimpar}

\begin{align*} %\label{sjsfw2}
\ww_p\cdot \ww_q=\alpha_0^{p,q} +\sum_{\substack{{K}\\{|K|=r-1}}}
\alpha^{p,q}_{(0,K)}W_K +
\sum_{\substack{{K}\\{|K|=r-n}\\ n \textnormal{ odd}}}
\alpha^{p,q}_{(1,K)}V_{(1)}W_K, \textnormal{ where } \alpha_0^{p,q}\in\{0,1\}
\end{align*}

\end{lemma}
\begin{proof} 
 Since $\ww_p\cdot\ww_q\in\mathcal{E}_0$, 
 we can assume that there exist some scalars 
 $\alpha_{(0,K)}^{p,q}$, $\alpha_{(I,K)}^{p,q}$ 
 such that 
\begin{align}\label{sjsfw}
\ww_p\cdot \ww_q=\alpha_0^{p,q} +\sum_{\substack{{I,K}\\{|\eta_{I,K}|=r-1}\\ {|K| \textnormal{ even}}}}
\alpha^{p,q}_{(I,K)}V_IW_K+  \sum_{\substack{{I,K}\\{\eta_{I,K}=r-1}\\ {|K| \textnormal{ even}}}}
\alpha^{p,q}_{(I,K)}V_IW_K,
\end{align} 
where 
$\alpha_0^{p,q}$ is $0$ or  $1$.

\medskip If we take
$a_i=\ww_p$, $a_j=\ww_q$ and $a_k=a_l=\wv_i$ in \eqref{idsupjordan}, 
and use Lemma \ref{Afirmacion1}, then 
we obtain 
$((\ww_p\cdot \ww_q)\cdot \wv_i)\cdot \wv_i=\ww_p\cdot \ww_q.$ 
Combining \eqref{sjsfw} in the stated before equality, 
we have
\begin{align}\label{jordan3forma}
\sum_{\substack{{I,K}\\{|\eta_{I,K}|=r-1}\\ {|K| \textnormal{ even}}}}
\alpha^{p,q}_{(I,K)}(\eta_{I,K}\cdot \wu_i)\cdot \wv_i &+  \sum_{\substack{{I,K}\\{\eta_{I,K}=r-1}\\ {|K| \textnormal{ even}}}}
\alpha^{p,q}_{(I,K)}(\eta_{I,K}\cdot \wu_i)\cdot \wv_i =\nonumber\\ 
& 
\sum_{\substack{{I,K}\\{|\eta_{I,K}|=r-1}\\ {|K| \textnormal{ even}}}}
\alpha^{p,q}_{(I,K)}\eta_{I,K}+  \sum_{\substack{{I,K}\\{\eta_{I,K}=r-1}\\ {|K| \textnormal{ even}}}}
\alpha^{p,q}_{(I,K)}\eta_{I,K}.
\end{align}

 Let 
 $\eta=\eta_{I,K}$
  be a nonzero element in 
  $\rad_0$. 
  Using equality \eqref{accionpar}, 
  one can easily prove the following statments

\noindent\textbf{ (i)} 
If $|\eta_{I,K}|=r-1$, 
and 
$i_j=0$,
then
$(\eta\cdot \wv_j)\cdot \wv_j=\eta$.\\
\noindent\textbf{ (ii)} 
If 
$|\eta_{I,K}|=r-1$, 
and 
$i_j=1$,
then 
$\eta\cdot \wv_j=0$ 
therefore, 
$(\eta\cdot \wv_j)\cdot \wv_j=0$.\\
\noindent\textbf{ (iii)} 
If 
$|\eta_{I,K}|=r$, 
and 
$i_j=0$, 
then 
$(\eta\cdot \wv_j)=0$, 
thus
$(\eta\cdot \wv_j)\cdot \wv_j=0$.\\
\noindent\textbf{ (iv)} 
If 
$|\eta_{I,K}|=r$, 
and 
$i_j=1$, 
then 
$(\eta\cdot \wv_j)\cdot \wv_j=\eta$.

\medskip 
Using statments 
 \textbf{(i)} - \textbf{(iv)}, 
we note that if 
$|\eta|=r-1$ and $i_j=1$ 
for some 
$j$ 
then the left part of \eqref{jordan3forma} is equal to zero 
and, consequently, the right part is zero.
Thus, in the right part of \eqref{jordan3forma} 
the only terms of length 
$r-1$ 
are of type 
$w_1^{k_1}\cdots w_{2m}^{k_{2m}}$.
Now, if 
 $|\eta|=r$ and $i_j=0$ 
 for some
  $j$ then, 
  $\wv_j\cdot \eta=0$.  
  Hence, every term of length 
  $r$ 
  on the right hand side of \eqref{jordan3forma} must contain every 
  $v_j$, 
  but this is only possible if 
  $n$ 
  is an odd integer. 

We have thus proved 
\begin{align*} %\label{sjsfw2}
\ww_p\cdot \ww_q=\alpha_0^{p,q} +\sum_{\substack{{K}\\{|K|=r-1}}}
\alpha^{p,q}_{(0,K)}W_K
\sum_{\substack{{K}\\{|K|=r-n}\\ n \textnormal{ odd}}}
\alpha^{p,q}_{(1,K)}V_{(1)}W_K.
\end{align*}\end{proof}

\begin{lemma}\label{lemaexistenciau}  
 \begin{align*}
 \ww_p=w_p+\sum_{\substack{{K}\\{|K|=r}}}
\xi^{p}_{(0,K)}W_K+\sum_{\substack{{K}\\{|K|=r-n-1}\\ n\textnormal{ odd}}}
\xi^{p}_{(1,K)}V_{(1)}W_K.
\end{align*}
\end{lemma}
\proof
 Let 
 $p\in\{1,\ldots,2m\}$ 
  be a fixed integer. We assume that there exist some scalars 
 $\xi_{(I,K)}^p$ 
 such that 
 \begin{align}\label{sjsfexuc}
\ww_p=w_p+\sum_{\substack{{I,K}\\{|\eta_{I,K}|=r-1}\\ {K \textnormal{ odd}}}}
\xi^{p}_{(I,K)}\eta_{I,K} +
\sum_{\substack{{I,K}\\{|\eta_{I,K}|=r}\\ {K \textnormal{ odd}}}}
\xi^{p}_{(I,K)}\eta_{I,K}.
\end{align}
Using Lemma 
\ref{Afirmacion1},
we have that 
$\ww_p\cdot\wv_j=0$, 
and therefore,
 \begin{align}\label{sjsfexu}
0=\sum_{\substack{{I,K}\\{|\eta_{I,K}|=r-1}\\ {K \textnormal{ odd}}}}
\xi^{p}_{(I,K)}\eta_{I,K}\cdot \wv_j +
\sum_{\substack{{I,K}\\{|\eta_{I,K}|=r}\\ {K \textnormal{ odd}}}}
\xi^{p}_{(I,K)}\eta_{I,K}\cdot \wv_j.
\end{align}
 Let 
 $\eta=\xi_{I,K}\eta_{I,K}$ 
 be a nonzero element in \eqref{sjsfexu}.
  We shall  use atatments \textbf{(i)} - \textbf{(iv)} 
from Lemma \ref{lemaformaimpar}. 

\noindent We note that If  
 $|\eta|=r-1$  
 and 
 $i_j=1$, 
 then 
 $\eta\cdot \wv_j=0$. 
Using  \eqref{sjsfexu}  one can easily verify that if 
$I\neq (1)$ 
then 
 $\xi_{(I,K)}=0$.
 Therefore, we see that the only elements in 
 \eqref{sjsfexuc} 
 of lenght 
 $r-1$
 are of type  
 $V_{(1)}W_K$. 
  
\noindent
Let
 $|\eta|=r$ 
 and 
 $i_j=1$, 
 then 
 $(\eta\cdot \wv_j)\cdot \wv_j=\eta$ 
 and therefore,  if 
 $V_I\neq 1$, 
 then
 $\xi_{(I,K)}=0$. 
 Thus, the only elements of lenght 
 $r$ 
  that are not  zero 
  on the right part of 
  \eqref{sjsfexu} 
  are precisely those where 
  $i_j=0$. 
  As this is valid for every 
  $j$, 
  we have that the only elements of lenght 
  $r$ 
  that appear in 
  \eqref{sjsfexuc} 
  are of type 
  $w_1^{k_1}\cdots w_{2m}^{k_{2m}}$, with 
  $|K|=r$. 
  Thus, we have proved that
\begin{equation*}
\ww_p=w_p+\sum_{\substack{{K}\\{|K|=r}}}
\xi^{p}_{(0,K)}W_K+\sum_{\substack{{K}\\{|K|=r-n-1}\\ n\textnormal{ odd}}}
\xi^{p}_{(1,K)}V_{(1)}W_K.
\end{equation*}
\vspace{-2mm}
\endproof 

\subsection{
$\rad$  be isomorphic to $u\,\algc_r/u\,\algc_{r-2}$,  where 
$r$  is
an even integer and 
$u$ is
an even vector.} 
Without loss of generality we can take
 \begin{align*}
\rad_0&= \textnormal{{\bf{\vect}}}_\corf\,\langle\, uV_IW_K\, |\eta_{I,K}|=r,r-1 \text{ and } |K| \text{ even } \rangle, \\
\rad_1&= \textnormal{{\bf{\vect}}}_\corf\,\langle\, uV_IW_K, |\eta_{I,K}|=r,r-1 \text{ and } |K| \text{ odd } \rangle.
 \end{align*} 
 As in above case, one can easily verify that
\begin{align*}
uV_IW_K\circ \wv_j=
(-\frac{1}{2})^{i_1+\cdots +i_{j-1}}uV_{(i_1,\ldots,i_j+1,\ldots,i_n)}W_K %v_1^{i_1}\cdots v_j^{i_j+1}\cdots v_n^{i_n}
((-1)^{|\eta_{I,K}|-i_j}-1).
\end{align*} 
Moreover, we have 
\begin{eqnarray*}
\begin{aligned}
uV_IW_K\cdot \wv_j &= \left\{\begin{array}{lll}
\pm uV_{(i_1,\ldots,i_{j-1}, 1, \ldots, i_n)}W_K & \textnormal{ if } & |\eta_{I,K}|=r-1, \, i_j=0,\\ %v_1^{i_1}\cdots v_j\cdots v_n^{i_n}
\pm uV_{(i_1,\ldots,i_{j-1}, 1, \ldots, i_n)}W_K & \textnormal{ if } & |\eta_{I,K}|=r, \, i_j=1,\\ %v_1^{i_1}\cdots v_j^{0}\cdots v_n^{i_n}
\end{array}\right. \end{aligned}  \label{accionparrpar}\\ 
\begin{aligned}
uV_IW_K\cdot \ww_p&=\left\{
\begin{array}{lll}
\mp k_{p\pm 1}uV_IW_{(k_1,\ldots,k_{p\pm1}-1,\ldots,k_{2m})}& \textnormal { if } & |\eta_{I,K}|=r, \\ 
%w_1^{k_1}\cdots w_{p\pm 1}^{k_{p\pm 1}-1}\cdots w_{2m}^{k_{2m}}
uV_IW_{(k_1,\ldots,k_{p}+1,\ldots,k_{2m})} &\textnormal { if }& |\eta_{I,K}|=r-1.  %w_1^{k_1}\cdots w_p^{k_p+1}\cdots w_{2m}^{k_{2m}} 
\end{array}\right.
\end{aligned} \label{accionimparrpar}
\end{eqnarray*}
Now, we note that there exist analogues to Lemmas \ref{Afirmacion1}, \ref{lemaformaimpar} and  \ref{lemaexistenciau}, implying the following equalities
\begin{align*}
\ww_{p}\cdot \ww_{q}&=\delta_{p+1,q}+\sum_{\substack{{K}\\{|K|=r}}}
\alpha^{p,q}_{(0,K)}uW_{(k_1,\ldots,k_{2m})}+
\sum_{\substack{{K}\\{|K|=r-n}\\ n\textnormal { even}}}
\alpha^{p,q}_{(1,K)}uV_{(1)}W_{(k_1,\ldots,k_{2m})},\\
\ww_p&=w_p+\sum_{\substack{{K}\\{|K|=r-1}}}
\xi^{i}_{(0,K)}uW_{(k_1,\ldots,k_{2m})}+\sum_{\substack{{K} \\{|K|=r-n-1}\\ n\textnormal { even}}}
\xi^{i}_{(1,K)}uV_{(1)}W_{(k_1,\ldots,k_{2m})}.
\end{align*}

We shall prove 
the following theorem 
\begin{theorem}\label{casosupfor}
Let
$\alga$
be a finite-dimensional Jordan superalgebra,
$\rad$
be the solvable radical of
$\alga$
such that
$\radd=0$
and
$\alga/\rad$
is isomorphic to the Jordan superalgebra of superform
$\jor$. 
Then $\alga\cong\alga/\rad\oplus\rad$ if and only if, 
$\rad\in\overline{\mathfrak{M}}(\alga/\rad; J^{(k)})$ 
where 
$J^{(k)}=\algc_{2k+1}/\algc_{2k-1}$ if $\dim V_0=2k+1$ 
or 
$J^{(k)}=a\algc_{2k}/a\algc_{2k-2}$ if $\dim V_0=2k$. 
\end{theorem}
\begin{proof} 
By Theorem \ref{teoredu} it suffices to prove the theorem when 
$\rad$ 
is irreducible. So, by Theorem 7.7 in \cite{zelmar2} 
we only  need to consider the two cases. 

\noindent 
 Using Lemma \ref{Afirmacion1}, 
we have 
$\ww_p\cdot\wv_i=0$ 
for 
$i=1,\ldots, n$; $p=1,\ldots, 2m$.

Let 
$p$ 
be an odd integer.  
Due to  Lemmas 
\ref{lemaformaimpar} and  \ref{lemaexistenciau}, 
we can assume that
\begin{align}\label{zpzqbase}
\ww_{p}\cdot \ww_{q}=\delta_{p+1,q}+\sum_{\substack{{K}\\{|K|=r-1}}}
\alpha^{p,q}_{(0,K)}W_K+
\sum_{\substack{{K}\\{|K|=r-n}\\ n\textnormal { odd}}}
\alpha^{p,q}_{(1,K)}V_{(1)}W_K, \\
\ww_p=w_p+\sum_{\substack{{K}\\{|K|=r}}}
\xi^{p}_{(0,K)}W_K+\sum_{\substack{{K} \\{|K|=r-n-1}\\ n\textnormal { odd}}}
\xi^{p}_{(1,K)}V_{(1)}W_K.
\end{align}
Thus
\begin{equation}\label{zpzqprod}
\begin{aligned} 
\ww_p\cdot\ww_q=& w_p\cdot w_q +\sum_{\substack{{K}\\{|K|=r}}}
\xi^{p}_{(0,K)}W_K\cdot w_q+\sum_{\substack{{K} \\{|K|=r-n-1}\\ n\textnormal { odd}}}
\xi^{p}_{(1,K)}V_{(1)}W_k\cdot w_q +\\%\nonumber\\
&\sum_{\substack{{K}\\{|K|=r}}} \xi^{p}_{(0,K)}w_p\cdot W_K +
\sum_{\substack{{K} \\{|K|=r-n-1}\\ n\textnormal { odd}}} \xi^{p}_{(1,K)}w_p\cdot V_{(1)}W_K.
\end{aligned}
\end{equation}
 Using  
 \eqref{generalwimpar}, \eqref{zpzqbase} and \eqref{zpzqprod}, 
 we have that 
 $\ww_p\cdot\ww_q=\delta_{p+1,q}$ 
 if and only if,
\begin{equation}\label{probandosf}
\begin{aligned}
0=&
\sum_{\substack{{K}\\{|K|=r-1}}}
\alpha^{p,q}_{(0,K)}W_K+ 
\sum_{\substack{{K}\\{|K|=r-1}}}
(-k_{p+1}\xi^{p}_{(0,K)})W_{(k_1,\ldots,k_{p+1}-1,\ldots,k_{2m})} +\\ 
&
\sum_{\substack{{K}\\{|K|=r-1}}}
k_{p-1}\xi^{q}_{(0,K)}W_{(k_1,\ldots,k_{p-1}-1,\ldots,k_{2m})} %w_1^{k_1}\cdots w_{p-1}^{k_{p-1}-1}\cdots w_{2m}^{k_{2m}}
+
 \sum_{\substack{{K}\\{|K|=r-n}\\ n\textnormal { odd}}}
\alpha^{p,q}_{(1,K)}V_{(1)}W_K +\\
&
\sum_{\substack{{K} \\{|K|=r-n}\\ n\textnormal { odd}}}
\xi^{p}_{(1,K)}V_{(1)}W_{(k_1,\ldots,k_{p+1}-1,\ldots,k_{2m})} %w_1^{k_1}\cdots w_p^{k_p+1}\cdots w_{2m}^{k_{2m}}) 
+\sum_{\substack{{K} \\{|K|=r-n}\\ n\textnormal { odd}}}
\xi^{q}_{(1,K)}V_{(1)}W_{(k_1,\ldots,k_{p}+1,\ldots,k_{2m})}. %w_1^{k_1}\cdots w_p^{k_p+1}\cdots w_{2m}^{k_{2m}})
\end{aligned}
\end{equation}
 Combining the above equality with a linear independence property of the elements
 $V_IW_K$ 
we have the following relations:
\begin{equation}\label{SSPrm1}
\begin{aligned}
\sum_{\substack{{K}\\{|K|=r-1}}}\alpha^{p,q}_{(0,K)}W_K
=&\sum_{\substack{{K} \\{|K|=r-1}}}k_{p+1}\xi^{p}_{(0,K)}W_{(k_1\ldots,k_{p+1}-1,\ldots,k_{2m})} -\\%\nonumber\\
%w_1^{k_1}\cdots w_{p+1}^{k_{p+1}-1}\cdots w_{2m}^{k_{2m}} 
&\sum_{\substack{{K}\\{|K|=r-1}}}
(k_{p-1}\xi^{q}_{(0,K)})W_{(k_1\ldots,k_{p-1}-1,\ldots,k_{2m})}, %w_1^{k_1}\cdots w_{p-1}^{k_{p-1}-1}\cdots w_{2m}^{k_{2m}}\Big] 
\end{aligned}
\end{equation}

\begin{equation}\label{SSPr}
\begin{aligned}
\sum_{\substack{{K}\\{|K|=r-n}\\ n\textnormal { odd}}}\alpha^{p,q}_{(1,K)}V_{(1)}W_K &=
-\sum_{\substack{{K} \\{|K|=r-n}\\ n\textnormal { odd}}}
\xi^{p}_{(1,K)}V_{(1)} W_{(k_1\ldots,k_{p}+1,\ldots,k_{2m})} -\\ %\nonumber\\%w_1^{k_1}\cdots w_p^{k_p+1}\cdots w_{2m}^{k_{2m}}\\ +&
& \sum_{\substack{{K} \\{|K|=r-n}\\ n\textnormal { odd}}}
\xi^{q}_{(1,K)}V_{(1)}W_{(k_1\ldots,k_{p}+1,\ldots,k_{2m})}.
\end{aligned}
\end{equation}
 Let
 $\alpha_{(0,S_{t})}^{p,q}W_{S_{t}}$ be a nonzero element
 at the left part of \eqref{SSPrm1}, 
 such that
 $S_t=(s_1,\ldots,s_{p-1},s_p,s_{p+1},\ldots,s_n)$
is a 2m-tupla, with  
 $|S_t|=r-1$. 
 We shall find a 2m-tupla 
 $S_{p+1}$ and $S_{p-1}$ 
 on the right part of  \eqref{SSPrm1}, such that 
 $S_{t+1}=(s_1,\ldots,s_{p-1},s_p,$ $s_{p+1}+1,\ldots,s_n) $ 
 and 
 $S_{t-1}=(s_1,\ldots,s_{p-1}+1,s_p,s_{p+1},\ldots,s_n)$. 
 We observe that  $|S_{t-1}|=|S_{t+1}|=r$.
 
 \medspace
Applying similar arguments to above stated, and using \eqref{SSPr}, 
 we have that for each
 $K_t=(k_1,\ldots,k_p,\ldots,k_n)$ 
 we should take 
 $\wt{K}_t=(k_1,\ldots,k_p-1,\ldots,k_n)$.  
Moreover, if 
$|K_t|=r-n$, 
then 
$|\wt{K}_t|=r-n-1$.

\medspace It is easy to see that the equations  \eqref{SSPrm1} and  \eqref{SSPr} are respectively equivalent to 
\begin{equation}\label{ayudar} 
\begin{aligned}
&\sum_{\substack{{K}\\{|K|=r-1}}}
\Big(\alpha^{p,q}_{(0,S)}-(s_{p+1}+1)\xi^{p}_{(0,S_{p+1})}+(s_{p-1}+1)\xi^{q}_{(0,S_{p-1})}\Big)W_{S_t} =0,\\
& \sum_{\substack{{K}\\{|K|=r-n}\\ n\textnormal { odd}}}
\Big(\alpha^{p,q}_{(1,K_p)}+\xi^{q}_{1,\wt{K}_p}+\xi^{p}_{1,\wt{K}_p}\Big)V_{(1)}W_{K_t}=0.
\end{aligned}
\end{equation}
Using the linear independance of 
$W_K$, $V_{(1)}W_K$
and  \eqref{ayudar}, for each 
$t\in\{1,\ldots, 2m\}$,
we have
\begin{equation}
\begin{aligned}\label{prosol}
\alpha^{p,q}_{(0,S_t)}-(s_{p+1}+1)\xi^{p}_{(0,S_{t+1})}+(s_{p-1}+1)\xi^{q}_{(0,S_{t-1})} = 0, \\
\alpha^{p,q}_{(1,K_t)}+\xi^{q}_{(1,\wt{K}_t)}+\xi^{p}_{(1,\wt{K}_t)} =0.
\end{aligned}  
\end{equation} 
Hence, we have a solvable linear equation system if $r\neq n$. 
We note that an analogous procceding is valid if $r$ is an even integer.
\end{proof}

\begin{remark}
By Lemmas \ref{lemaformaimpar} and \ref{lemaexistenciau}, if 
$n$ 
is an odd integer and 
$r=n$, 
then
\begin{equation*}
\ww_p=w_p+\sum_{\substack{{K}\\{|K|=r}}}
\xi^{p}_{(0,K)}W_K, \quad
 \textnormal{ and }\quad 
{\ww}_{p}\cdot {\ww}_{q}=\delta_{i+1,j}+\sum_{\substack{{K}\\{|K|=n-1}}}
\alpha^{p,q}_{(0,K)}W_K+
\alpha^{p,q}_{(1,0)}V_{(1)}.
\end{equation*}
We see that the system
$\alpha^{p,q}_{(0,S_t)}-\xi^{p}_{(0,S_{t+1})}+\xi^{q}_{(0,S_{t-1})}=0$, and  $\alpha^{p,q}_{(1,0)} =0$
 has no solution when $\alpha^{p,q}_{(1,0)}\neq 0$.
\end{remark}

%RC

\subsection{counter-examples to WPT for Jordan superalgebras of superform with radical $\algc_r/\algc_{r-2}$ and $\dim\jor_0=r$}
Now we will show that the restrictions imposed in the Theorem \ref{casosupfor} are essential, and we have two cases to consider:

\noindent {\bf{Case 1.}} 
Let $n$ be an odd integer. Consider the  superalgebra 
$$\jor=(\corfd 1+\corfd v_1+\cdots+\corfd v_n+\rad_0)\dotplus (\corfd w_1+\cdots+\corfd w_{2m}+\rad_1).$$ 
where 
\begin{equation*}
\begin{aligned}
\rad_0=\textrm{Spann}\langle\, v_1^{i_1}\cdots v_n^{i_n}w_1^{k_1}\cdots w_{2m}^{k_{2m}}, \quad |K| \textnormal{ is  even}, \mid I\mid +\mid K\mid=n-1 \text{ or } n \,\rangle, \\
\rad_1=\textrm{Spann}\langle\, v_1^{i_1}\cdots v_n^{i_n}w_1^{k_1}\cdots w_{2m}^{k_{2m}}, \quad |K| \textnormal{ is odd },
\mid I\mid +\mid K\mid=n-1 \text{ or } n  \,\rangle
\end{aligned}
\end{equation*}
 where $i_1,\ldots,i_n$ are $0$ or $1$  and $k_i$ are nonnegative integers, 
 $|K|=k_1+\cdots+k_{2m}$, $|I|=i_1+\cdots+i_n$.
All nonzero products of the basis elements of $\jor$ are defined as follows
\begin{equation*}
\begin{aligned}
v_i ^2=1,  
w_{1}\cdot w_{2}=1+v_1\cdots v_n=-w_2\cdot w_1,  \\
 w_{2s-1}\cdot w_{2s}=- w_{2s}\cdot w_{2s-1}=1\, \textnormal{ for } s\in\{2,3,\ldots,m\}, 
\end{aligned}
\end{equation*}
\begin{align*}
v_1^{i_1}\cdots v_n^{i_n}w_1^{k_1}\cdots w_{2m}^{k_{2m}}\cdot w_{p}=\frac{1}{2}v_1^{i_1}\cdots v_n^{i_n}w_1^{k_1}\cdots w_{p}^{k_{p}+1}\cdots w_{2m}^{k_{2m}}(1+(-1)^{|I|+|K|})-\nonumber\\ 
k_{p+1}v_1^{i_1}\cdots v_n^{i_n}w_1^{k_1}\cdots w_{p+1}^{k_p-1}\cdots w_{2m}^{k_{2m}} \textnormal{ if } p=2s-1,\, s\in\{1,\ldots,m\},
\end{align*}
\begin{align*}
v_1^{i_1}\cdots v_n^{i_n}w_1^{k_1}\cdots w_{2m}^{k_{2m}}\cdot w_{p}=\frac{1}{2}v_1^{i_1}\cdots v_n^{i_n}w_1^{k_1}\cdots w_{p}^{k_{p}+1}\cdots w_{2m}^{k_{2m}}(1+(-1)^{|I|+|K|})+\nonumber\\ 
k_{p-1}v_1^{i_1}\cdots v_n^{i_n}w_1^{k_1}\cdots w_{p-1}^{k_{p-1}-1} \cdots w_{2m}^{k_{2m}} \textnormal{ if } p=2s,\, s\in\{1,\ldots,m\},
\end{align*}
\begin{align*}
v_1^{i_1}\cdots v_n^{i_n}w_1^{k_1}\cdots w_{2m}^{k_{2m}}\cdot v_{j}&=\nonumber\\  
(-\frac{1}{2})^{i_1+\cdots+i_{j-1}}v_1^{i_1}\cdots &v_j^{i_j+1}\cdots  v_n^{i_n}w_1^{k_1}\cdots w_{2m}^{k_{2m}}(1+(-1)^{|I|+|K|-i_j}).
\end{align*}
%\begin{align*}
%v_i ^2=1,  
%w_{1}\cdot w_{2}=1+v_1\cdots v_n=-w_2\cdot w_1,  \\
% w_{2s-1}\cdot w_{2s}=- w_{2s}\cdot w_{2s-1}=1\, \textnormal{ for } s\in\{2,3,\ldots,m\}, \\
%v_1^{i_1}\cdots v_n^{i_n}w_1^{k_1}\cdots w_{2m}^{k_{2m}}\cdot w_{p}=\frac{1}{2}v_1^{i_1}\cdots v_n^{i_n}w_1^{k_1}\cdots w_{p}^{k_{p}+1}\cdots w_{2m}^{k_{2m}}(1+(-1)^{|I|+|K|})-\\ 
%k_{p+1}v_1^{i_1}\cdots v_n^{i_n}w_1^{k_1}\cdots w_{p+1}^{k_{p}-1}\cdots w_{2m}^{k_{2m}} \textnormal{ if } p=2s-1,\, s\in\{1,\ldots,m\},\\
%v_1^{i_1}\cdots v_n^{i_n}w_1^{k_1}\cdots w_{2m}^{k_{2m}}\cdot w_{p}=\frac{1}{2}v_1^{i_1}\cdots v_n^{i_n}w_1^{k_1}\cdots w_{p}^{k_{p}+1}\cdots w_{2m}^{k_{2m}}(1+(-1)^{|I|+|K|})+\\ 
%k_{p-1}v_1^{i_1}\cdots v_n^{i_n}w_1^{k_1}\cdots w_{p-1}^{k_{p-1}-1} \cdots w_{2m}^{k_{2m}} \textnormal{ if } p=2s,\, s\in\{1,\ldots,m\},\\
%v_1^{i_1}\cdots v_n^{i_n}w_1^{k_1}\cdots w_{2m}^{k_{2m}}\cdot v_{j}=\\  
%(-\frac{1}{2})^{i_1+\cdots+i_{j-1}}v_1^{i_1}\cdots v_j^{i_j+1}\cdots  v_n^{i_n}w_1^{k_1}\cdots w_{2m}^{k_{2m}}(1+(-1)^{|I|+|K|-i_j}).
%\end{align*}

\medskip
We note that $\jor/\rad=(\corfd 1+\corfd v_{1}+\cdots+\corfd v_n)\dotplus(\corfd w_1+\cdots+\corfd w_{2m})$ is a Jordan  superalgebra isomorphic to Jordan superalgebra of superform, $\rad$ is isomorphic to $\algc_{n}/\algc_{n-2}$.

\medskip
If we assume that the WPT is valid for $\jor$, then, for $i=1,\ldots, 2m$ there exists $\wt{w}_{i} \in\jor_1$ such that $\wt{w}_{i}\equiv w_i (\textrm{mod}\rad_1)$, and  $\wt{w}_{2i-1}\cdot\wt{w}_{2i}=1$.

\medskip 
\noindent By Lemma \ref{lemaexistenciau} there exist 
 $\beta_{K}, \,\xi_{K},\,\alpha,\,\lambda\in\corf$ 
 such that 
 $\displaystyle{\wt{w}_{1}=w_{1}+\sum_{|K|=n}\beta_{K}w_1^{k_{1}}\cdots w_{2m}^{k_{2m}}}$ and 
$\displaystyle{\wt{w}_{2}=w_{2}+\sum_{|K|=n}\xi_{K}w_1^{k_{1}}\cdots w_{2m}^{k_{2m}}}$. 
Hence,
 \begin{align*}
 \wt{w}_{1}\cdot\wt{w}_{2}=w_{1}\cdot w_{2}+\sum_{|K|=n}\xi_{k}w_1\cdot w_1^{k_{1}}\cdots w_{2m}^{k_{2m}}+\sum_{|K|=n}\beta_{K}w_1^{k_{1}}\cdots w_{2m}^{k_{2m}}\cdot w_2.
 \end{align*}

\noindent We observe that 
 $\wt{w}_{1}\cdot\wt{w}_{2}=1$
 if and only if 
 $\displaystyle{v_1\cdots v_n+\sum_{|K|=n}\omega_{K}w_1^{t_{1}}\cdots w_{2m}^{t_{2m}}}=0$. 
 Using the fact that 
  $v_1\cdots v_n$ and 
  $ w_1^{k_{1}}\cdots w_{2m}^{k_{2m}}$ 
  are linearly independent, we have a contradiction.

\medskip
\noindent{\bf{Case 2}} 
Let $n$ be an even integer. Consider the  superalgebra 
$$\jor=(\corfd 1+\corfd v_1+\cdots+\corfd v_n+\rad_0)\dotplus (\corfd w_1+\cdots+\corfd w_{2m}+\rad_1).$$ 

$\rad_0$ is spanned by 
$\langle\, uv_1^{i_1}\cdots v_n^{i_n}w_1^{k_1}\cdots w_{2m}^{k_{2m}}, \quad |K| \textnormal{ is  even} \,\rangle$ 
and 
$\rad_1$ is spanned by 
$\langle\, uv_1^{i_1}\cdots v_n^{i_n}w_1^{k_1}\cdots w_{2m}^{k_{2m}}, \quad |K| \textnormal{ is odd} \,\rangle$,
 where $i_1,\ldots,i_n$ are $0$ or $1$  and $k_i$ are nonnegative integers, 
 $|K|=k_1+\cdots+k_{2m}$, $|I|=i_1+\cdots+i_n$ and $|K|+|I|=n$ or 
 $|K|+|I|=n-1$.

All nonzero products of the basis elements of $\jor$ are defined as follows
\begin{align*}
v_i ^2&=1, \quad 
w_{1}\cdot w_{2}=1+uv_1\cdots v_n=-w_2\cdot w_1,\\
 w_{2s-1}\cdot w_{2s}&=- w_{2i}\cdot w_{2i-1}=1\, \textnormal{ for } s\in\{2,3,\ldots,m\}, 
\end{align*}
\begin{align*}
uv_1^{i_1}\cdots v_n^{i_n}w_1^{k_1}\cdots w_{2m}^{k_{2m}}\cdot w_{p}&=\frac{1}{2}uv_1^{i_1}\cdots v_n^{i_n}w_1^{k_1}\cdots w_{p}^{k_{p}+1}\cdots w_{2m}^{k_{2m}}(1+(-1)^{|I|+|K|})+\\ 
k_{p+1}uv_1^{i_1}\cdots v_n^{i_n}&w_1^{k_1}\cdots w_{p+1}^{k_{p+1}-1}\cdots w_{2m}^{k_{2m}},\quad  \textnormal{ if } p=2s-1,\, s\in\{1,\ldots,m\},
\end{align*}
\begin{align*}
uv_1^{i_1}\cdots v_n^{i_n}w_1^{k_1}\cdots w_{2m}^{k_{2m}}\cdot w_{p}&=\frac{1}{2}uv_1^{i_1}\cdots v_n^{i_n}w_1^{k_1}\cdots w_{p}^{k_{p}+1}\cdots w_{2m}^{k_{2m}}(1+(-1)^{|I|+|K|})-\\ 
k_{p-1}uv_1^{i_1}\cdots v_n^{i_n}&w_1^{k_1}\cdots w_{p-1}^{k_{p-1}-1} \cdots w_{2m}^{k_{2m}},\quad  \textnormal{ if } p=2s,\, s\in\{1,\ldots,m\},
\end{align*}
\begin{align*}
uv_1^{i_1}\cdots v_n^{i_n}w_1^{k_1}\cdots w_{2m}^{k_{2m}}\cdot v_{j}&=\\  
(-\frac{1}{2})^{i_1+\cdots+i_{j-1}+1}&uv_1^{i_1}\cdots v_j^{i_j+1}\cdots  v_n^{i_n}w_1^{k_1}\cdots w_{2m}^{k_{2m}}(1+(-1)^{|I|+|K|-i_j}).
\end{align*}

\medskip
 It is easy to verify  that $\jor/\rad$ is a Jordan superalgebra of superform  and $\rad\cong u\,\algc_{n}/u\,\algc_{n-2}$.

 If we assume that the WPT is valid for $\jor$, then, for $i=1,\ldots, 2m$ there exists $\wt{w}_{i} \in\jor_1$ such that $\wt{w}_{i}\equiv w_i(\textrm{mod}\rad_1)$, and  $\wt{w}_{2i-1}\cdot\wt{w}_{2i}=1$, $i\geq 2$.
 
 \medskip
 By an analogous  to Lemma \ref{lemaexistenciau}, we have that 
 $\displaystyle{\wt{w}_{1}=w_{1}+\sum_{|K|=n-1}\beta_{K}uw_1^{k_{1}}\cdots w_{2m}^{k_{2m}}}$ and 
 $\displaystyle{\wt{w}_{2}=w_{2}+\sum_{|K|=n-1}\xi_{K}uw_1^{k_{1}}\cdots w_{2m}^{k_{2m}}}$ for some  
 $\beta_{K}, \,\xi_{K},\,\alpha,\,\lambda\in\corf$.
 
 It is clear that 
 $\displaystyle{\wt{w}_{1}\cdot\wt{w}_{2}=w_{1}\cdot w_{2}+\sum_{|K|=n-1}\omega_{K}uw_1\cdot w_1^{t_{1}}\cdots w_{2m}^{t_{2m}}}$,
 $\omega_K\in\corf$.
Therefore 
 $\wt{w}_{1}\cdot\wt{w}_2=1$
 if and only if 
$\displaystyle{uv_1\cdots v_n+\sum_{|K|=n-1}\omega_K uw_1\cdot w_1^{k_{1}}\cdots \ w_{2m}^{k_{2m}}}=0$.
Once again, 
 $uv_1\cdots v_n$ and  $uw_1^{k_{1}}\cdots w_{2m}^{k_{2m}}$ are linearly independent, consequently, we have a contradiction.

 %%%%%%%%%%%%%%%%%%%%%%%%%%%%%%%%%%%%%%%%%%%%%%%%%%%%%%%%%%%%%%%%%%%%
 \section{Superalgebra $\algdt$ and Kaplansky $\kap$}
 %RC:
 In this section, we  consider the Jordan superalgebra
 $\algdt=(\corfd e_1+\corfd e_2)\dotplus(\corfd x+\corfd y)$,
 and Kaplansky, 
 $\kap=(\corfd e)\dotplus(\corfd x+\corfd y)$.
 
 We stress that 
 $\algdt$ 
 is a simple Jordan superalgebra if 
 $t\neq 0$. If $t=0$, then $\algdc$ contain $\kap$.
 Unital irreducible superbimodules over 
 $\algdt$ and $\kap$ 
 were classified by C.~Martinez and E.~Zelmanov in~\cite{zelmar3}
 and by M.~Trushina in~\cite{Tru}.
 In this section, we shall use the examples, 
 notations and ideas introduced  by M. Trushina.

 \medskip
 %RC
 Let $\sld$ be a Lie algebra with the  basis
 $e,f,h$
 and the multiplication given by 
 $[f,h]=2f,\quad [e,h]=2e,\quad [e,f]=h, $ where $[a,b]=ab-ba$.
 
 We shall say that a module $\biml$ with the basis
 $l_0,l_1,\ldots,l_n$
 is an irreducible
 $sl_2$-module with standard basis
 $l_0,l_1,\ldots,l_n$
 if
 \begin{equation}
 \begin{array}{lll}\label{sl2modulo}
 &l_i\cdot h=(n-2i)l_i,&\nonumber\\
 l_0\cdot e=0, &\, l_i\cdot e=(-in+i(i-1))l_{i-1} \,& \textnormal{ for } i\,>\,0,\nonumber\\
 l_n\cdot f=0, &\, l_i\cdot f=l_{i+1} \, &\textnormal{ for } i\,<\,n.
 \end{array}
 \end{equation}

 By $\mro_a$ we denote the operator of right multiplication by $a$, 
 we also denote it by the capital  letter $A$. 
 One  can easily check that the operators 
 $\frac{2}{1+t}X\circ Y, \frac{2}{1+t}X^2, \,\frac{2}{1+t}Y^2$ 
 span the simple lie algebra $\sld$.
 In terms of operators above, it is easy to see that a superbimodule 
 $\biml$ with basis 
 $l_0,l_1,\ldots,l_n$ 
 is an irreducible $sl_2$-module with the standard basis 
 $l_0,l_1,\ldots,l_n$ if
 \begin{equation}\label{basestandard}
 \begin{aligned}
 l_i&\,X\circ Y=\frac{1+t}{2}(n-2i)l_{i}, \\
 l_0&\,X^2=0, \quad l_i\,X^2=\frac{1+t}{2}(-in+i(i-1)l_{i-1}) \textnormal{ for any } i\,>\, 0,\\
 l_n&\,Y^2=0, \quad l_i\,Y^2=\frac{1+t}{2}l_{i+1} \textnormal{ for } i\,<\, n.
 \end{aligned}
 \end{equation}
 
 %RC
 In terms of right multiplication operators, equality  \eqref{idsupjordan} may be written as follows:
 \begin{equation}\label{supjorop}
 \begin{aligned}
 \mro_{a_i}\mro_{a_j}\mro_{a_k}+&(-1)^{ij+ik+jk}\mro_{a_k}\mro_{a_j}\mro_{a_i}+(-1)^{jk}\mro_{(a_ia_k)a_j}=\\
 &\mro_{a_i}\mro_{a_ja_k}+(-1)^{ij+ik+jk}\mro_{a_k}\mro_{a_ja_i}+(-1)^{ij}\mro_{a_j}\mro_{a_ia_k}.
 \end{aligned}
 \end{equation}

 %RC
 Substituting 
 $a_i=a_k=x$ and $a_j=e_1$ 
 (respectively $a_i=a_k=y$ and $a_j=e_1$ ) in \eqref{supjorop}, we obtain 
 $[X^2,E]=0$ (respectively  $ [Y^2,E]=0$), where $E$ denote $\mro_{e_1}$.

\subsection{Jordan superalgebra $\algdt$} 
%Here, $\algs_0=\corfd\we+\corfd\wf\cong(\algdt)_0$ and 
%$\alga_1/\rad_1=\corfd\bar{x}+\corfd\bar{y}\cong(\algdt)_1$, 
%$\alga_0=\algs_0\oplus\rad_0$.

In this section,  we shall prove the following theorem.
 
 \begin{theorem}\label{TeoDt}
 Let 
 $\alga$ be a finite-dimensional Jordan superalgebra with a solvable radical 
 $\rad$  such that $\radd=0$ 
and 
 $\alga/\rad\cong\algdt$, $t\neq -1$.
 Then $\alga\cong\alga/\rad\oplus\rad$ Iff 
 $\rad\in\overline{\mathfrak{M}}(\alga/\rad;J^{(k)})$  where $J^{(k)}$ is a 
 regular superbimodule.
 \end{theorem}
 
 \begin{proof}
 Using \ref{teoredu} and the Theorem 1.1 in \cite{Tru}, we need to consider three main cases.
Here, $\algs_0=\corfd\we+\corfd\wf\cong(\algdt)_0$ and 
$\alga_1/\rad_1=\corfd\bar{x}+\corfd\bar{y}\cong(\algdt)_1$, 
$\alga_0=\algs_0\oplus\rad_0$.

\medskip \textbf{Case 1}
   {Let $n$ be a positive integer and suppose 
   $t\in\mathbb{R}$, $t\neq 0,1,-\frac{n+2}{n}$.}
   
   \noindent We assume that 
   where 
   $ \rad_0=\biml_{n+1}^1\oplus\biml_{n+1}^2$, 
   $\rad_1=\biml_{n+2}\oplus\biml_n$. 
   Here,  
   $\biml_{n+1}^1,\,\biml_{n+1}^2,\, \biml_{n+2}$, 
   $\biml_n$  are the same as in Example  1 in \cite{Tru}. 
   
It is easy to see that 
    $E\mid_{\rad_1}\equiv \frac{1}{2}$, 
  therefore
    $\wt{e}_i\cdot \wx=\frac{1}{2}\wx$ 
    and $\wt{e}_i\cdot \wy=\frac{1}{2}\wy$. 
  Assume that there exist  scalars 
  $\beta^{x}_{3,i}$, $\beta^{x}_{4,i}$, $\beta^{y}_{3,i}$, $\beta^{y}_{4,i}$, 
  $\xi_{1,j}^{x,y}$ and $\xi_{2,j}^{x,y}$
  for 
  $i=-1, 0, 1, \ldots, n$ and $j=0,1,\ldots, n$, 
  such that 
  \begin{align}
  \begin{aligned}\label{dtwtuz}
  \wx=& x+\beta_{3,-1}^{x}m+\beta_{3,n}^{x}mY^{2(n+1)}+\sum_{k=1}^{n}\gamma_{3,4,k}^{x}mY^{2k}-\sum_{k=1}^{n}\beta_{4,k}^{x}mY^{2k-1}EY,\\
  \wy=& y+\beta_{3,-1}^{y}m+\beta_{3,n}^{y}mY^{2(n+1)}+\sum_{k=1}^{n}\gamma_{3,4,k}^{y}mY^{2k}-\sum_{k=1}^{n}\beta_{4,k}^{y}mY^{2k-1}EY,
  \end{aligned} \\
  \begin{aligned}\label{dtuz}
    x\cdot y=\we+t\wf+ \sum_{k=0}^{n}(\xi_{1,k}^{x,y}-\xi_{2,k}^{x,y})mY^{2k+1}E+\sum_{k=0}^{n}\xi_{2,k}^{x,y}mY^{2k+1}.
  \end{aligned}
  \end{align} 
  where 
  $\gamma_{3,4,k}^{x}=\beta_{3,k-1}^{x}+\alpha\beta_{4,k}^{x}$ 
  and 
  $\gamma_{3,4,k}^{y}=\beta_{3,k-1}^{y}+\alpha\beta_{4,k}^{y}$.
    
  Now, we have 
  $\wx\cdot\wy=\we+t\wf$, 
  if and only if
  \begin{equation}\label{cincouno}
  \begin{aligned}
  0=&
  \Big(\beta_{3,-1}^{x}m+\beta_{3,n}^{x}mY^{2n+2}
  +\sum_{k=1}^{n}\gamma_{3,4,k}^{x}mY^{2k}-\sum_{k=1}^{n}\beta_{4,k}^{x}mY^{2k-1}EY\Big)\cdot y-\\
  &
  \Big(\beta_{3,-1}^{y}m+\beta_{3,n}^{y}mY^{2n+2}
  +\sum_{k=1}^{n}\gamma_{3,4,k}^{y}mY^{2k}-\sum_{k=1}^{n}\beta_{4,k}^{y}mY^{2k-1}EY\Big)\cdot x +\\
  & 
  \sum_{k=0}^{n}(\xi_{1,k}^{x,y}-\xi_{2,k}^{x,y})mY^{2k+1}E+\sum_{k=0}^{n}\xi_{2,k}^{x,y}mY^{2k+1}=\\
  &\beta_{3,-1}^{x}mY+\beta_{3,n}^{x}mY^{2n+3}
    +\sum_{k=1}^{n}\gamma_{3,4,k}^{x}mY^{2k+1}
    +\sum_{k=1}^{n}\beta_{4,k}^{x}mY^{2k+1}E+\\
    &\beta_{3,n}^{y}\frac{(1+t)(n+1)}{2}mY^{2n+1}
    +\sum_{k=1}^{n}\frac{(1+t)(n-(k-1))}{2}\beta_{4,k}^{y}mY^{2k-1}E + \\
    &\sum_{k=1}^{n}\Big(\frac{1+t}{2}k\gamma_{3,4,k}^{y}+(-1)^{k}\frac{(1+t)n+2}{4}\beta_{4,k}^{y}\Big)mY^{2k-1} +\\ 
    &\sum_{k=1}^{n+1}(\xi_{1,k-1}^{x,y}-\xi_{2,k-1}^{x,y})mY^{2k-1}E
    +\sum_{k=1}^{n+1}\xi_{2,k-1}^{x,y}mY^{2k-1}.
  \end{aligned}
  \end{equation}
  
%  Using %\eqref{baseMn} and the relation 
%  $[Y^2,E]=0$, 
%  we have that \eqref{cincouno} is  equivalent to the following equality
%  
%  \begin{equation}\label{515}
%  \begin{aligned}
%  0=&\beta_{3,-1}^{x}mY+\beta_{3,n}^{x}mY^{2n+3}
%  +\sum_{k=1}^{n}\gamma_{3,4,k}^{x}mY^{2k+1}
%  +\sum_{k=1}^{n}\beta_{4,k}^{x}mY^{2k+1}E+\\
%  &\beta_{3,n}^{y}\frac{(1+t)(n+1)}{2}mY^{2n+1}
%  +\sum_{k=1}^{n}\frac{(1+t)(n-(k-1))}{2}\beta_{4,k}^{y}mY^{2k-1}E + \\
%  &\sum_{k=1}^{n}\Big(\frac{1+t}{2}k\gamma_{3,4,k}^{y}+(-1)^{k}\frac{(1+t)n+2}{4}\beta_{4,k}^{y}\Big)mY^{2k-1} +\\ 
%  &\sum_{k=1}^{n+1}(\xi_{1,k-1}^{x,y}-\xi_{2,k-1}^{x,y})mY^{2k-1}E
%  +\sum_{k=1}^{n+1}\xi_{2,k-1}^{x,y}mY^{2k-1}.
%  \end{aligned}
%  \end{equation}
  Since $\alpha=\frac{(1+t)n+2}{2(1+t)(n+1)}$,
   we have that
    \begin{equation}\label{eqppalk}
  \begin{aligned}
 & \left[\,\beta_{3,-1}^{x}+\frac{(1+t)\beta_{3,0}^{y}}{2}-\frac{n\alpha(1+t)\beta_{4,1}^{y}}{2}+\xi_{2,0}^{x,y}\,\right]\,mY+\beta_{3,n}^{x}mY^{2n+3}+\\
 & \left[\,\beta_{3,n-1}^{x}+\alpha\beta_{4,n}^{x}+\xi_{2,n}^{x,y}+\frac{(1+t)(n+1)\beta_{3,n}^{y}}{2}\right]mY^{2n+1} +\\
 &  \left[\frac{(1+t)n\beta_{4,1}^{y}}{2}+\xi_{1,0}^{x,y}-\xi_{2,0}^{x,y}\,\right]\,mYE +
  \left[\xi_{1,n}^{x,y}-\xi_{2,n}^{x,y}+\beta_{4,n}^{x}\right]mY^{2n+1}E +\\
 & \sum_{k=2}^{n}\left[\,\beta_{3,k-2}^{x}+\alpha\beta_{4,k-1}^{x}+\frac{(1+t)}{2}[k\beta_{3,k-1}^{y}+
 \alpha(k+(-1)^k(n+1)\beta_{4,k}^{y}]+\xi_{2,k-1}^{x,y}\,\right]\,mY^{2k-1}+\\
  &\sum_{k=2}^{n}\left[\,\beta_{4,k-1}^{x}+\beta_{4,k}^{y}\frac{(1+t)(n-(k-1))}{2}+\xi_{1,k-1}^{x,y}-\xi_{2,k-1}^{x,y}\,\right]\,mY^{2k-1}E=0.
  \end{aligned}
  \end{equation}
  
  \par\medskip
  %RC
 Note that fixing the 
  $\xi$'s 
  in \eqref{dtuz}, we get
  %$\beta_{3,-1}^{u}, \beta_{3,1}^{u}, \ldots,\beta_{3,n-1}^{u}$ 
  \begin{equation}\label{eq1-3k}
  \begin{aligned}
  \beta_{3,n}^{x}&=0,\quad \beta_{4,n}^{x}=\xi_{2,n}^{x,y}-\xi_{1,n}^{x,y},\quad \beta_{4,1}^{y}=\frac{2(\xi_{2,0}^{x,y}-\xi_{1,0}^{x,y})}{(t+1)n},\\
  \beta_{3,0}^{x}&=\frac{n\alpha(1+t)\beta_{4,1}^{y}-2\beta_{3,-1}^{x}-2\xi_{2,0}^{x,y}}{(1+t)},\\
  \beta_{3,n}^{y}&=\frac{-2(\beta_{3,n-1}^{x}+\alpha\beta_{4,n}^{x}+\xi_{2,n}^{x,y})}{(1+t)(n+1)},\\
  \beta_{4,k}^{y}&=\frac{2}{(1+t)(n-(k-1))}\left[-\beta_{4,k-1}^{x}-\xi_{1,k-1}^{x,y}+\xi_{2,k-1}^{x,y}\right], \quad k=2,\ldots,n\\
  \beta_{3,k-1}^{y}&=\frac{1}{k}\left[\frac{-2(\beta_{3,k-2}^{x}+\alpha\beta_{4,k-1}^{x}+\xi_{2,k-1}^{x,y})}{(1+t)}-\alpha(k+(-1)^k(n+1)\beta_{4,k}\,\right]
  \end{aligned}
  \end{equation} 
  such that the equality \eqref{eqppalk} holds.
   
  \medskip {\bf{Case 2.}}
 Let $n$ be a positive integer, 
 $\frac{1}{t}=-\frac{n}{n+2}$. 
 Consider the following cases:
 $\rad\cong\bim(n+1,n+2)$, 
 $\rad\cong \bim(n+1,n)$,  
 $\rad\cong\widetilde{\bim}(n_1)$ and $n_1\neq n$. 
 (See example 2 in \cite{Tru}.) 
 
 \medskip {\bf{(A)}} 
Assume that 
 $\rad\cong\bim(n+1,n+2)$ 
 where 
 $\rad_0$
  is the irreducible $sl_2$-module with the standard basis 
  $l_0,\ldots,l_n$ and 
  $\rad_1$ is spanned by 
  $l_0x,l_0y,l_1y,\ldots,l_ny$. 
 
 \noindent
 Since $E\mid_{\bim_1}\equiv\frac{1}{2}$, then  $\wt{e}_i\cdot \wx=\frac{1}{2}\wx$ and $\wt{e}_i\cdot\wy=\frac{1}{2}\wy$.
 Assume that there exist 
 $\xi_0^{u,z},\ldots, \xi_n^{u,z}$,
  $\beta_{0,x}^{u}$, $\beta_{0}^{u}$,$ \ldots$, $\beta_{n}^{u}$, 
  $\beta_{0,x}^{z}$, $\beta_{0}^{z}$, $\ldots$, and $\beta_{n}^{z}$ 
scalars such that 
\begin{equation*}
\begin{aligned}
x\cdot y&=\we+t\wf+\sum_{k=0}^{n}\xi_k^{x,y}l_k, \\
\wx&=x+\beta_{0,x}^xl_0x+\sum_{k=0}^{n}\beta_{k}^{x}l_ky, \quad  
\wy=y+\beta_{0,x}^yl_0x+\sum_{k=0}^{n}\beta_{k}^{y}l_ky
\end{aligned}
\end{equation*}
 
 We observe that 
 $\wx\cdot\wy=\we+t\wf$ 
 if and only if
 \begin{align*}
 0&=\sum_{k=0}^{n}\xi_k^{x,y}l_k+    \beta_{0,x}^{x}\,\left(\frac{1-t}{2}\right)\,l_0+\beta_n^{y}\,\left(\frac{(n+1)t}{n+2}\right)\,l_n\nonumber\\
 &+\sum_{k=0}^{n-1}\beta_k^{x}\,\left(\frac{1+t}{2}\right)\,l_{k+1}-\sum_{k=0}^{n-1}\beta_k^{y}\,\left(\frac{1+t}{2(n-k)}\right)(-(k+1)n+(k+1)k)l_k,
 \end{align*}
 which gives rise to the system of equations
 \begin{equation}\label{sistemadt2}
 \begin{aligned}
 0&=\xi_0^{x,y}+\beta_{0,x}^{x}\,\left(\frac{1-t}{2}\right)+\beta_0^{y}\,\left(\frac{1+t}{2}\right),\\ 0&=\xi_n^{x,y}+\beta_n^{y}\,\left(\frac{(n+1)t}{n+2}\right)+\beta_{n-1}^{x}\,\left(\frac{1+t}{2}\right), \\
 0&=\xi_k^{x,y}+\beta_{k-1}^{x}\,\left(\frac{1+t}{2}\right)-\beta_k^{y}\,\left(\frac{1+t}{2(n-k)}\right)(-(k+1)n+(k+1)k)l_k,
 \end{aligned}
 \end{equation} 
 for  $ k=1,2,\ldots,n-1$. 
 We note that the system \eqref{sistemadt2} has always a solution.
 
 \medskip {\bf{(B)}} 
 $\rad\cong\bim(n+1,n)$ where  
 $\rad_0$ is the irreducible $sl_2$-module with the standard basis 
 $l_0,\ldots,l_n$ and  
 $\rad_1$ is spanned by $l_1x,\ldots,l_nx$.

 As in the case \textbf{(A)}, 
 $\wt{e}_{i}\cdot \wx=\frac{1}{2}\wx$, 
 $\wt{e}_{i}\cdot\wy=\frac{1}{2}\wy$ 
 and there exist 
 $\xi_0^{x,y}, \ldots, \xi_n^{x,y}\in\corf$ 
 such that 
 $\displaystyle{x\cdot y=\we+t\wf+\sum_{k=0}^{n}\xi_k^{x,y}l_k}.$
Let us find  
 $\beta_{k}^{x}$ and $\beta_{k}^{y}$ 
 such that 
 $ \wx=x+\sum_{k=1}^{n}\beta_{k}^{x}l_kx$, and 
  $\wy=y+\sum_{k=1}^{n}\beta_{k}^{y}l_kx$. 
 
 Now, we note that 
 $\wx\cdot\wy=\we+t\wf$
  if and only if
 \begin{align*}
 0=\sum_{k=0}^{n}\xi_k^{x,y}l_k-\sum_{k=1}^{n}\beta_k^{x}\, \left(\frac{1+t}{2}\right)k\,l_{k}-\sum_{k=1}^{n}\beta_k^{y}\,\left(\frac{1+t}{2}\right)(-kn+k(k-1))l_{k-1},
 \end{align*}
 which gives rise to the system of equations
 
 \begin{equation}\label{sistemadt3}
 \begin{aligned}
 0=&\quad\xi_0^{x,y}-\beta_{0}^{y}\,\left(\frac{1-t}{2}\right)n =\xi_n^{x,y}-n\beta_n^{x}\,\left(\frac{1+t}{2}\right),\\
 0=&\quad\xi_k^{x,y}-\beta_{k}^{x}k\,\left(\frac{1+t}{2}\right)-\beta_{k+1}^{y}\,\left(\frac{1+t}{2}\right)(-(k+1)n+(k+1)k),
 \end{aligned}
 \end{equation} 
 for $ k=1,2,\ldots,n-1$.
 
 %\medskip 
 \noindent System of equations \eqref{sistemadt3}  has always a solution. 
 
 \medskip {\bf{(C)}} 
Finally, if 
 $\rad\cong\wt{\bim}(n_1)$, $n_1\neq n$. 
 This case is similar to Case \textbf{(1)}, with the replacement of $n$ by $n_1$.

 \par \medskip 
 {\bf{Case 3.}}
 Let $t=1$.
 In this case, 
 $\rad$ 
 is isomorphic to
 $\wt{\bim}(n)$ 
 or to a 1-dimensional vector space with a generator $m$ such that $mx=mx=0$, $me=\frac{1}{2}m$. (See example 3 in \cite{Tru}.) 
 
 \medspace It only remains to consider the case when $\rad$ is isomorphic to a $1$-dimensional vector space.
 The case when 
 $\rad\cong\wt{\bim}(n)$ 
 is similar to Case \textbf{(1)}. In particular we take 
 $t=1$ 
 in the equation \eqref{eqppalk}.
 Now we shall consider two subcases:
 
 \medskip {\bf{(A)}} 
 If $m$ is an even vector, then  
 $\rad_0=\corfd m$.
Assume that 
 $x\cdot y=\we+t\wf+\eta m$, for some $\eta\in\corf$
 Since $\rad_1=0$ 
 we have 
 $\wx=x ,\, \wy=y $. 
Note that the equality 
 $\wx\cdot\wy=\we+t\wf$ 
holds if and only if $\eta=0$. 
If we take $a_i=x$, $a_j=y$, $a_k=e_1$ and $a_l=e_2$ in the equalitie \eqref{idsupjordan} we obtained, 
$0=((\widetilde{x}\cdot \widetilde{y})e_1)e_2=((e_1+te_2+\eta m)e_1)e_2=\frac{1}{4}\eta m$ 
and therefore, $\eta=0$, thus the WPT is valid.

 \medskip {\bf{(B)}} 
 If $m$ is odd vector. 
 In this case 
 $\rad_0=0$  and $\rad_1=\corfd m$.
 Therefore,
 $x\cdot y=\we+t\wf$. 
 Let $\wx=x+\beta^um$ and  $\wy=y+\beta^zm $, 
 hence $\wx\cdot\wy=\we+t\wf$ is always solvable.
 
 \medskip 
 From Case \textbf{(1)} - \textbf{(3)}, we conclude  that it is possible to give some conditions for
 $\eta_u$ and $\eta_z\in\rad_1$, 
 such that an analogue to WPT is valid under the Theorem conditions. 
 \end{proof}

\subsection{Jordan superalgebra $\kap$.} We shall proof the following theorem
\begin{theorem}\label{casokap}
 Let 
 $\alga$ 
 be a finite-dimensional Jordan superalgebra with solvable radical 
 $\rad,$ $ \radd=0$ 
 and such that 
 $\alga/\rad\cong \kap$. 
 Then there exists a subsuperalgebra 
 $\algs\subseteq\alga$ 
 such that  
 $\alga/\rad\cong\algs$ and 
 $\alga=\algs\oplus\rad$.
 \end{theorem}
 \begin{proof} Since Theorem \ref{teoredu} and \cite{Tru}, we have to consider two cases to know, 
 $\rad\cong\reg \mathcal{K}_3$  and $\rad\cong\wt{\bim}(n)$.  
 But the second cases is analogous to  case \textbf{(1)}, for  
 $\algdt$, 
  one can obtain an analogue of equality \eqref{cincouno} substituting 
  $t=0$. This gives rise to the system of equations equivalent to \ref{eq1-3k}. 
  
  Therefore, we consider $\rad\cong\reg \mathcal{K}_3$.  
  Assume that 
  $(\kap)_0\cong\algs_0=\corfd\we$ and $(\kap)_1\cong\alga_1/\rad_1=\corfd\bar{x}+\corfd\bar{y}$, and 
  $\rad=\corfd f\dotplus(\corfd u+\corfd z)$, 
  where
  $f\leftrightarrow e$, $u\leftrightarrow x$,  $z\leftrightarrow y$. 
  Let $\wx$ and $\wy$ be some preimages of
 $\bar{x}$ and $\bar{y}$ respectively and suppose that $xy=\we+\eta f$ for some $\eta\in\corf$.
 Let $\alpha$, $\beta$, $\gamma$ and $\delta$ scalars such that $\wx=x+\alpha u+\beta z$ and
 $wy=y+\gamma u+\delta z$. We note that 
 $\wx\cdot \wy=\we$ if and only if $\alpha+\delta=\eta$.
  and  the equality is always solvable.
 \end{proof}
 
 \begin{remark}
 In the case of the Jordan superalgebra 
 $\kapun$ 
 we have that the irreducible superbimodules are the same as the ones for the Jordan superalgebra 
 $\kap$. 
 In general, for any algebra 
 $\alga$ 
 there exist an isomorphism of  category of superbimodules over 
 $\alga$, ($\textnormal{Bimod }\alga$) 
 into category of unital superbimodules over 
 $\alga^{\#}$, ($\textnormal{Bimod}\,\alga^{\#}$). 
 Thus, the proof of the above theorem  is also true if we substitute 
 $\kap$ by $\kapun$.
 \end{remark}

\subsection{Counter-examples to WPT for Jordan superalgebras of type $\algdt$, $t\neq-1$.}
Now we will show that  restrictions imposed in Theorem \ref{TeoDt}  are essential.

Let 
$\algb=\alga\oplus\rad$ 
be a superalgebra, where 
$\alga_0=\corfd e_1+\corfd e_2+\corfd a_1+\corfd a_2$, 
$\alga_1=\corfd x+\corfd y+\corfd v+\corfd w$,  
$\rad_0=\corfd a_1+\corfd a_2$ and 
$\rad_1=\corfd v+\corfd w$. 
All nonzero products of the basis elements of $\algb$ are defined as follows:
    \begin{equation}\label{dt1par}
    \begin{aligned}
    e_i^2=e_i, \quad e_ia_j=\delta_{ij}a_j, \quad e_ix=\frac{1}{2}x,\quad e_iy=\frac{1}{2}y,\\
    a_ix=\frac{1}{2}v,\quad a_iy=\frac{1}{2}w,\quad e_iv=\frac{1}{2}v,\quad e_iw=\frac{1}{2}w,
    \end{aligned}
    \end{equation} 
    \begin{equation}\label{dt1impar}
    \begin{aligned}
    xw=vy=a_1+t a_2,\quad  xy=e_1+t e_2+a_1+(-2-t)a_2
    \end{aligned}
    \end{equation} for $i=0,1$.
    The products in \eqref{dt1par} and \eqref{dt1impar} commute and anticommute respectively and $t\neq -1$.

   \medskip
   One easily verifies that $\algb$ is a Jordan superalgebra, and 
   $\algb/\rad$ 
   is a Jordan superalgebra isomorphic to 
   $\algdt$, $t\neq -1$,
 $\rad\cong\reg\algdt$ and $\radd=0$.
 
 Consider the product $xy=e_1+t e_2+\alpha a_1+\beta a_2$. 
 Replacing $a_i=a_k=x$, $a_j=y$ and $a_l=e_1$ in \eqref{idsupjordan} 
 we obtain 
 $((xy)\cdot x)\cdot e_1-\frac{1}{2}(xy)\cdot x=0$, 
 thus we have
 $1+t+\alpha+\beta=0$, later on $\alpha+\beta=-1-t$ 
 and therefore, 
 $\algb$ is a Jordan superalgebra.

\noindent 
If we assume that the WPT is valid for 
$\algb$, 
then there are
$\wt{x},\wt{y}$ 
such that 
$\wt{x}\equiv x$, 
$\wt{y}\equiv y \,(\textrm{mod }\rad_1)$  and 
$\wt{x}\wt{y}=e_1+t e_2$, $e_i\wt{x}=\frac{1}{2}\wt{x}$, $e_i\wt{y}=\frac{1}{2}\wt{y}$.

We note that $\wt{x}=x+\sigma v$ and $\wt{y}=y+\omega w$.
If $\wt{x}=x+\sigma v+\lambda w$, using $\wt{x}^2=0$, we obtain $\lambda xw=0$ and therefore
$\lambda=0$. Now
$$\wt{x}\wt{y}=xy+\sigma yv+\omega xw=e_1+t e_2+a_1+(-2-t)a_2-\sigma(a_1+ta_2)+\omega(a_1+ta_2)$$
Therefore, 
$1-\sigma+\omega=0$ and $(-2-t)-\sigma t+\omega t=0$, later on $\omega-\sigma=-1$
 and $0=(-2-t)+t(\omega-\sigma)=-2-2t$, thus $t=-1$
and this is a contradiction.

\section{Main theorem}
Using the Theorems \ref{casokac}, \ref{casosupfor}, \ref{TeoDt} and  \ref{casokap}, 
we have the following theorem:
\begin{theorem}
Let $\alga$ be a finite dimensional Jordan superalgebra with solvable radical $\rad$ such that 
$\radd=0$ and $\alga/\rad\cong\jor$ where 
$\jor$ is a simple Jordan superalgebra.
We set
$\mathfrak{M}(\jor;\rad_1,\ldots,\rad_t)=\{\,V/ V$
is a
$\jor$-superbimodule such that homomorphic images of
$V$
do not contain subsuperbimodules isomorphic to
$\rad_i$ for $i=1,2,\ldots,t\}$,  and 
$\overline{\mathfrak{M}}(\jor;\rad_1,\ldots,\rad_t)$ 
as an analogue of the class 
$\mathfrak{M}(\jor;\rad_1,\ldots,\rad_t)$ for irreducible $\jor$-superbimodules $\rad_i$.

If one of the following conditions holds:

\medskip\noindent{\bf{i)}}
{$\jor\cong\kac$;}

\medskip\noindent{\bf{ii)}} 
{$\jor\cong\kap$;}

\medskip\noindent{\bf{iii)}}
{\it{
$\jor$
is a superalgebra of a superform with even part of dimension $n$ such that \\
$\rad\in\mathfrak{M}(\jor;\algc_n/\algc_{n-2}\, (n  \textnormal{ is odd}),\,
u\cdot\algc_n/u\cdot \algc_{n-2}\,(n \textnormal{ is even}))$;}}

\medskip\noindent{\bf{iv)}}
{\it{
$\jor\cong\algdt$, $t \neq -1$,
$\rad\in\mathfrak{M}(\algdt;\reg\,\algdt)$;} }

\noindent{\it{then there is a subsuperalgebra
$\algs\subseteq\alga$
such that
$\algs\cong\jor$
and
$\alga=\algs\oplus\rad$, 
the restrictions of items iii) and iv) are essential.}}
\end{theorem}

% Bibliography
%-----------------------------------------------------------------

\bibliographystyle{amsplain}

\end{document}